\documentclass[12pt]{article}
\usepackage{mathrsfs}
\usepackage{graphicx}
\usepackage{subfig}
\usepackage{epstopdf}
\usepackage{amsmath}
\usepackage{amsfonts}
\usepackage{amssymb, amsmath, cite}
\usepackage{cases}

\newtheorem{theorem}{Theorem}
\newtheorem{lemma}{Lemma}

\newtheorem{assumption}{Assumption}

\newbox\qedbox
\newenvironment{proof}{\smallskip\noindent{\bf Proof.}\hskip \labelsep}%
                        {\hfill\penalty10000\copy\qedbox\par\medskip}

\newcommand{\bfR}{{\Bbb R}}
\newcommand{\bfC}{{\Bbb C}}

\newcommand{\ii}{\text{i}}
\newcommand{\e}{\text{e}}
\newcommand{\dd}{\text{d}}

\newcommand{\nn}{\nonumber}
\newcommand\be{\begin{equation}}
\newcommand\ee{\end{equation}}
\newcommand{\bea}{\begin{eqnarray}}
\newcommand{\eea}{\end{eqnarray}}
\newcommand\berr{\begin{eqnarray*}}
\newcommand\eerr{\end{eqnarray*}}
{\setlength\arraycolsep{2pt}

\begin{document}

\title{Long-time asymptotics for the Hirota equation on the half-line}
\author{ Boling Guo$^{a}$,\, Nan Liu$^{b,}$\footnote{Corresponding author.}\\
$^a${\small{\em Institute of Applied Physics and Computational Mathematics,  Beijing 100088, P.R. China}} \\
$^b${\small{\em The Graduate School of China Academy of Engineering Physics, Beijing 100088, P.R. China}}\setcounter{footnote}{-1}\footnote{E-mail addresses: gbl@iapcm.ac.cn (B. Guo), ln10475@163.com (N. Liu).}
}

\date{}
\maketitle

\begin{quote}
{{{\bfseries Abstract.} We consider the Hirota equation on the quarter plane with the initial and boundary values belonging to the Schwartz space. The goal of this paper is to study the long-time behavior of the solution of this initial-boundary value problem based on the asymptotic analysis of an associated matrix Riemann--Hilbert problem.

}

 {\bf Keywords:} Hirota equation,  Nonlinear steepest descent method, Long-time asymptotics.}
\end{quote}

\section{Introduction}
\setcounter{equation}{0}
The nonlinear steepest descent method was first introduced in 1993 by Deift and Zhou \cite{PD}, it turn out to be very successful for analysing the long-time asymptotics of initial-value problems for a large range of nonlinear integrable evolution equations in a rigorous and transparent form. Numerous new significant results about the asymptotics theory of initial-value problems for different completely integrable nonlinear equations were obtained based on the analysis of the corresponding Riemann--Hilbert (RH) problems \cite{AB1,AB2,AB3,RB,XJ1,XJ,HL}. After that, Fokas announced a new unified approach \cite{F1,F2} to construct the matrix RH problems for the analysis of initial-boundary value (IBV) problems for linear and nonlinear integrable systems. Therefore, by combining the ideas of the nonlinear steepest descent method in \cite{PD} with the unified transform formalism of \cite{F1}, it is also possible to study the asymptotics of solutions of IBV problems for nonlinear integrable PDEs such as the nonlinear Schr\"odinger (NLS) equation \cite{F3,AB}, the modified Korteweg-de Vries (mKdV) equation \cite{JL1}, the derivative Schr\"odinger equation \cite{JL2} and the Kundu--Eckhaus equation \cite{GL1}.

In 1973, Hirota \cite{RH} considered the following equation:
\be\label{1.2}
\ii\frac{\partial u}{\partial t}+\alpha\frac{\partial^2u}{\partial x^2}+\ii\beta\frac{\partial^3u}{\partial x^3}+3\ii\gamma|u|^2\frac{\partial u}{\partial x}+\delta|u|^2u=0,
\ee
where $u$ is a scale function, $\alpha,~\beta,~\gamma$ and $\delta$ are real positive constants with the relation $\alpha\gamma=\beta\delta$. Meanwhile, in \cite{RH}, the exact $N$-envelope-soliton solutions have been obtained there applying the method which takes his name \cite{RH1}. Equation \eqref{1.2} can be written as
\begin{equation}\label{1.1}
\ii u_t+\alpha(u_{xx}+2|u|^2u)+\ii\beta(u_{xxx}+6|u|^2u_x)=0,
\end{equation}
where we have chosen $\delta=2\alpha$, $\gamma=2\beta$ in such a way that the constraint $\alpha\gamma=\beta\delta$ is satisfied. Equation \eqref{1.1} is integrable because it is the sum of the commuting integrable flows given by the NLS equation and complex  mKdV which are PDEs belonging to the same hierarchy. On the other hand, due to the important role played in both physics and mathematics, Hirota equation has attracted much attention and various works were presented. For example, in \cite{TH}, the authors studied the multisolitons, breathers, and rogue waves for the Hirota equation \eqref{1.1} by means of the Darboux transformation. In a recent paper \cite{FD}, the explicit soliton solution formula for the Hirota equation have been constructed via the inverse scattering method and the matrix triplet method. However, it is note that the long-time asymptotics for the Hirota equation were analysed in \cite{HL} via nonlinear steepest descent method.

In this paper, we consider the IBV problem for the Hirota equation \eqref{1.1} posed on the quarter-plane domain, that is, in the domain
$$\Omega=\{(x,t)|0\leq x<\infty,~0\leq t<\infty\}.$$
We will denote the initial datum,  Dirichlet and Neumann boundary values of \eqref{1.1} as follows:
\bea
u(x,0)=u_0(x),~u(0,t)=g_0(t),~u_x(0,t)=g_1(t),~u_{xx}(0,t)=g_2(t).
\eea
We will assume $u_0(x),~g_0(t),~g_1(t),~g_2(t)$ lie in the Schwartz class $\mathcal{S}([0,\infty))$. As mentioned above, this IBV problem can be analyzed by the unified transformation approach presented in \cite{F1}. Indeed, assuming that the solution $u(x,t)$ exists and sufficiently smooth as well as rapidly decay as $x\rightarrow\infty$, one can show that it can be represented in terms of the solution of a $2\times2$ matrix RH problem formulated in the complex $k$-plane with the jump matrices given in terms of spectral functions $a(k)$, $b(k)$ obtained from the initial datum $u(x,0)=u_0(x)$ and $A(k)$, $B(k)$ obtained from the boundary values $u(0,t)=g_0(t)$, $u_x(0,t)=g_1(t)$ and $u_{xx}(0,t)=g_2(t)$. This parts have already been done in \cite{GL}.

Developing and extending the methods using in \cite{JL1,JL2}, our goal here is to explore the long-time asymptotics of the solution $u(x,t)$ for the Hirota equation \eqref{1.1} on the half-line. Compared with the analysis of the initial-value problem for \eqref{1.1}, the RH problem relevant for the IBV problem \cite{GL} also has jumps across additional two contours whereas the RH problem relevant for the initial-value problem only has a jump across $\bfR$. Moreover, the jumps across these additional two lines involve the spectral function $h(k)$. During the asymptotic analysis, one should find an suitable analytic approximation $h_a(t,k)$ of $h(k)$. Thus, the nonlinear steepest descent analysis of the half-line problem for \eqref{1.1} presents some additional challenges. On the other hand, for the Hirota equation \eqref{1.1}, the RH problem \cite{GL} relevant for the IBV problem has four jump matrices across the jump contours, however, there only contain three jump matrices for the mKdV equation \cite{AB4}. Another difference compared with the mKdV equation is that its spectral curve possesses two non-symmetric stationary points, which also different from that of derivative NLS equation considered in \cite{JL2} where the phase function has a single critical point. Therefore, the study of the long-time asymptotics for the IBV problem for \eqref{1.1} on the half-line is more involved. These are some innovation points of the present paper.

The organization of this paper is as follows. In Section 2, we briefly recall how the solution of Hirota equation \eqref{1.1} on the half-line can be expressed in terms of the solution of a $2\times2$ matrix RH problem. In Section 3, we transform the original RH problem to a form suitable for determining the long-time asymptotics. Local models for the RH problem near the two critical points are considered in Section 4. In Section 5, we derive the long-time asymptotic behavior of the solution of the Hirota equation \eqref{1.1} by combining the above analysis.
\section{Preliminaries}
\setcounter{equation}{0}
\setcounter{lemma}{0}
\setcounter{theorem}{0}
\subsection{The RH problem}
Let
\begin{eqnarray}
&&\sigma_3={\left( \begin{array}{cc}
1 & 0 \\
0 & -1\\
\end{array}
\right )},\quad
U(x,t)={\left( \begin{array}{cc}
0 & u \\
-u^* & 0 \\
\end{array}
\right )},\label{2.1}\\
&&V(x,t;k)=4\beta k^2U(x,t)+k V_1(x,t)+V_2(x,t),\label{2.2}\\
&&V_1(x,t)={\left( \begin{array}{cc}
2\ii\beta|u|^2  & 2\ii\beta u_x+2\alpha u\\
2\ii\beta u^*_x-2\alpha u^* & -2\ii\beta|u|^2 \\
\end{array}
\right )},\label{2.3}\\
&&V_2(x,t)={\left( \begin{array}{cc}
\ii\alpha|u|^2+\beta(uu_x^*-u^*u_x)  & \ii\alpha u_x-\beta(u_{xx}+2|u|^2u)\\
\ii\alpha u^*_x+\beta(u^*_{xx}+2|u|^2u^*) & -\ii\alpha|u|^2-\beta(uu_x^*-u^*u_x) \\
\end{array}
\right )}.\label{2.4}
\end{eqnarray}
Then the Hirota equation \eqref{1.1} is the condition of compatibility of \cite{TH}
\begin{equation}\label{2.5}
\begin{aligned}
&\Psi_x+\ii k[\sigma_3,\Psi]=U(x,t)\Psi,\\
&\Psi_t+\ii(4\beta k^3+2\alpha k^2)[\sigma_3,\Psi]=V(x,t;k)\Psi.\\
\end{aligned}
\end{equation}
where $\Psi(x,t;k)$ is a $2\times2$ matrix-valued function and $k\in\bfC$ is the spectral parameter. Following \cite{GL}, one can use the Lax pair \eqref{2.5} to define three suitable eigenfunctions to formulate the main RH problem. Furthermore, the solution $u(x,t)$ of Hirota equation \eqref{1.1} can be expressed in terms of the solution of this $2\times2$ matrix RH problem.

Assume the initial and boundary values $u_0(x),~g_0(t),~g_1(t),~g_2(t)$ lie in the Schwartz class $\mathcal{S}([0,\infty))$. Using these functions, we can define four spectral functions $\{a(k),~b(k),~A(k),~B(k)\}$ by
\be\label{2.6}
X(0,k)=\begin{pmatrix}
\overline{a(\bar{k})} & b(k)\\[4pt]
-\overline{b(\bar{k})} & a(k)\\
\end{pmatrix},\quad T(0,k)=\begin{pmatrix}
\overline{A(\bar{k})} & B(k)\\[4pt]
-\overline{B(\bar{k})} & A(k)\\
\end{pmatrix},
\ee
where $X(x,k)$ and $T(t,k)$ are the solutions of the Volterra integral equations
\bea
X(x,k)&=&I-\int_x^\infty\e^{\ii k(x'-x)\hat{\sigma}_3}\big(U(x',0)X(x',k)\big)\dd x',\label{2.7}\\
T(t,k)&=&I-\int_t^\infty\e^{\ii(4\beta k^3+2\alpha k^2)(t'-t)\hat{\sigma}_3}\big(V(0,t';k)T(t',k)\big)\dd t'.\label{2.8}
\eea
Define the open domain $\{D_j\}_1^6$ of the complex $k$-plane as depicted in Fig. \ref{fig1}. Let $$\Sigma=\{k\in\bfC|\text{Im}(4\beta k^3+2\alpha k^2)=0\}$$ be the contour separating the $D_j$'s, oriented as shown in Fig. \ref{fig1}, where $$k_0=-\frac{\alpha}{3\beta}.$$

\begin{figure}[htbp]
  \centering
  \includegraphics[width=3.5in]{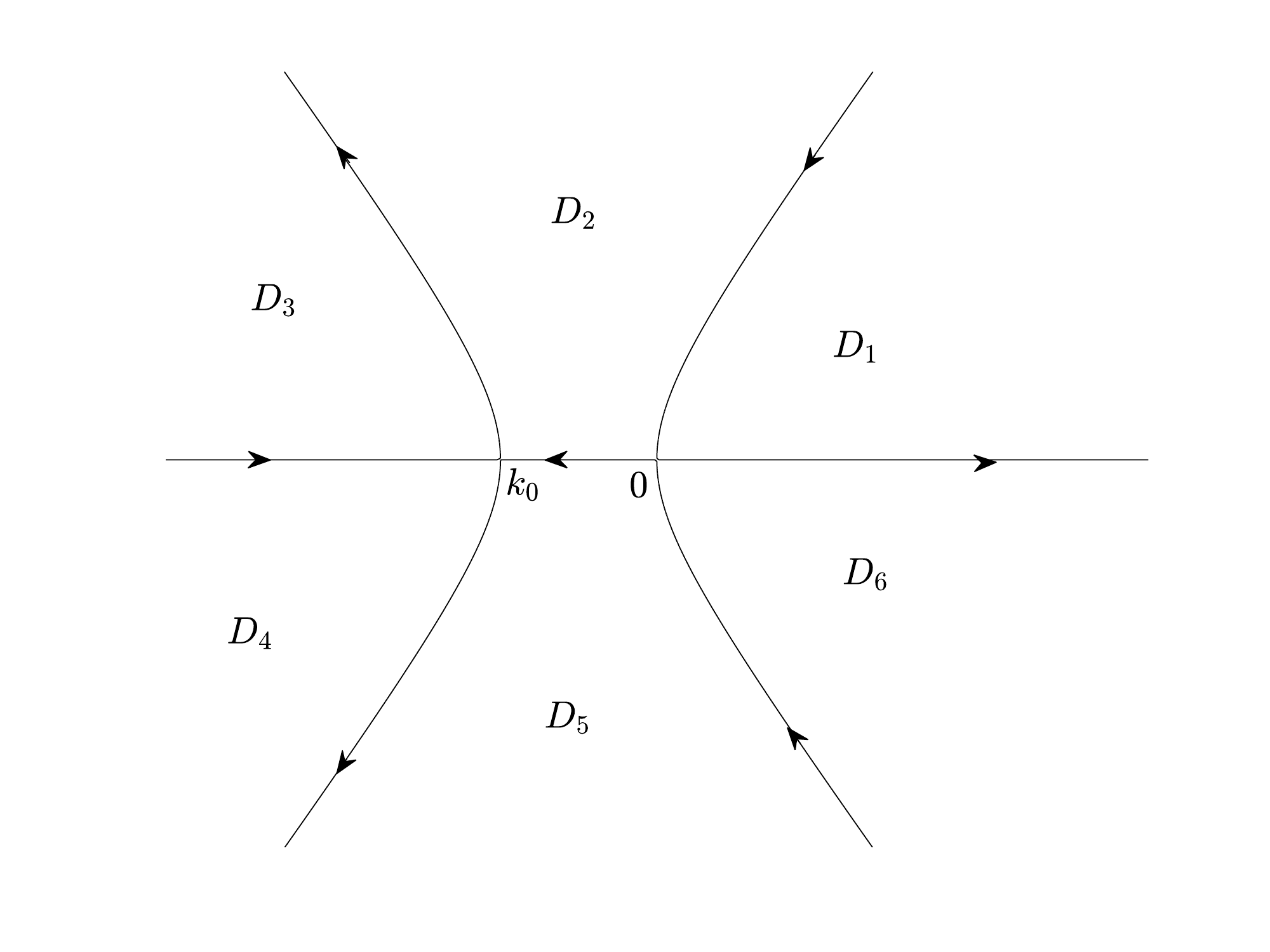}
  \caption{The contour $\Sigma$ and the domain $\{D_j\}_1^6$ in the complex $k$-plane.}\label{fig1}
\end{figure}

The analysis of the Volterra linear integral equations \eqref{2.7} and \eqref{2.8} give the following properties of $a(k)$, $b(k)$ and $A(k)$, $B(k)$.

(i) $a(k)$ and $b(k)$ are analytic for Im$k>0$, smooth and bounded for Im$k\geq0$;

(ii) $a(k)=1+O(\frac{1}{k})$, $b(k)=O(\frac{1}{k})$, as $k\rightarrow\infty$;

(iii) $a(k)\overline{a(\bar{k})}+b(k)\overline{b(\bar{k})}=1$ for $k\in\bfR$;

(iv) $A(k)$ and $B(k)$ are smooth and bounded for $k\in\bar{D}_1\cup \bar{D}_3\cup \bar{D}_5$, and analytic in $D_1\cup D_3\cup D_5$;

(v) $A(k)=1+O(\frac{1}{k})$, $B(k)=O(\frac{1}{k})$, $k\rightarrow\infty$;

(vi) $A(k)\overline{A(\bar{k})}+B(k)\overline{B(\bar{k})}=1$ for $k\in\Sigma$.

We will also need the spectral functions $c(k)$ and $d(k)$ defined by
\bea
c(k)&=&b(k)A(k)-a(k)B(k),\quad k\in\bar{D}_1\cup\bar{D}_3,\label{2.9}\\
d(k)&=&a(k)\overline{A(\bar{k})}+b(k)\overline{B(\bar{k})},\quad k\in \bar{D}_2,\label{2.10}
\eea
these functions are motivated by the fact that
$$T^{-1}(0,k)X(0,k)=\begin{pmatrix}
\overline{d(\bar{k})} & c(k)\\[4pt]
-\overline{c(\bar{k})} & d(k)\\
\end{pmatrix}.$$
The initial and boundary values $u_0(x),~g_0(t),~g_1(t),~g_2(t)$ can not be independently prescribed but must satisfy an important compatibility condition. This compatibility condition is conveniently expressed at the level of the spectral functions as the so-called global relation:
\be\label{2.11}
B(k)a(k)-A(k)b(k)=0,~k\in\bar{D}_1\cup \bar{D}_3.
\ee

\begin{assumption}\label{ass1}
In what follows, we assume that the following conditions hold:

$\bullet$  the initial and boundary values lie in the Schwartz class.

$\bullet$ the spectral functions $a(k), b(k), A(k), B(k)$ defined in \eqref{2.6} satisfy the global relation \eqref{2.11}.

$\bullet$ $a(k)$ and $d(k)$ have no zeros in $\bar{D}_1\cup \bar{D}_2\cup \bar{D}_3$ and $\bar{D}_2$, respectively.

$\bullet$ the initial and boundary values $u_0(x),~g_0(t),~g_1(t)$, and $g_2(t)$ are compatible with equation \eqref{1.1} to all orders at $x=t=0$, i.e., they satisfy
\bea
g_0(0)=u_0(0),\quad g_1(0)&=&u_0'(0),\quad g_2(0)=u''_0(0),\nn\\
\ii g'_0(0)+\alpha(2|u_0(0)|^2u_0(0)+u_{0}''(0))&+&\ii\beta(u'''_0(0)+6|u_0(0)|^2u'_0(0))=0,\cdots.\nn
\eea
\end{assumption}
Define the spectral functions $r_1(k)$ and $h(k)$ by
\bea
r_1(k)&=&\frac{\overline{b(\bar{k})}}{a(k)},\qquad\quad~~ k\in\bfR,\label{2.12}\\
h(k)&=&-\frac{\overline{B(\bar{k})}}{a(k)d(k)},\quad k\in\bar{D}_2,\label{2.13}
\eea
and let $r(k)$ denote their sum given by
\be\label{2.14}
r(k)=r_1(k)+h(k)=\frac{\overline{c(\bar{k})}}{d(k)},\quad k\in\bfR.
\ee
Then, it can show that the following RH problem
\be\label{2.15}
M_+(x,t;k)=M_-(x,t;k)J(x,t,k),\quad \mbox{Im}(4\beta k^3+2\alpha k^2)=0,
\ee
with the jump matrix $J(x,t,k)$ depicted in Fig. \ref{fig2} is defined by
\begin{figure}[htbp]
  \centering
  \includegraphics[width=3.5in]{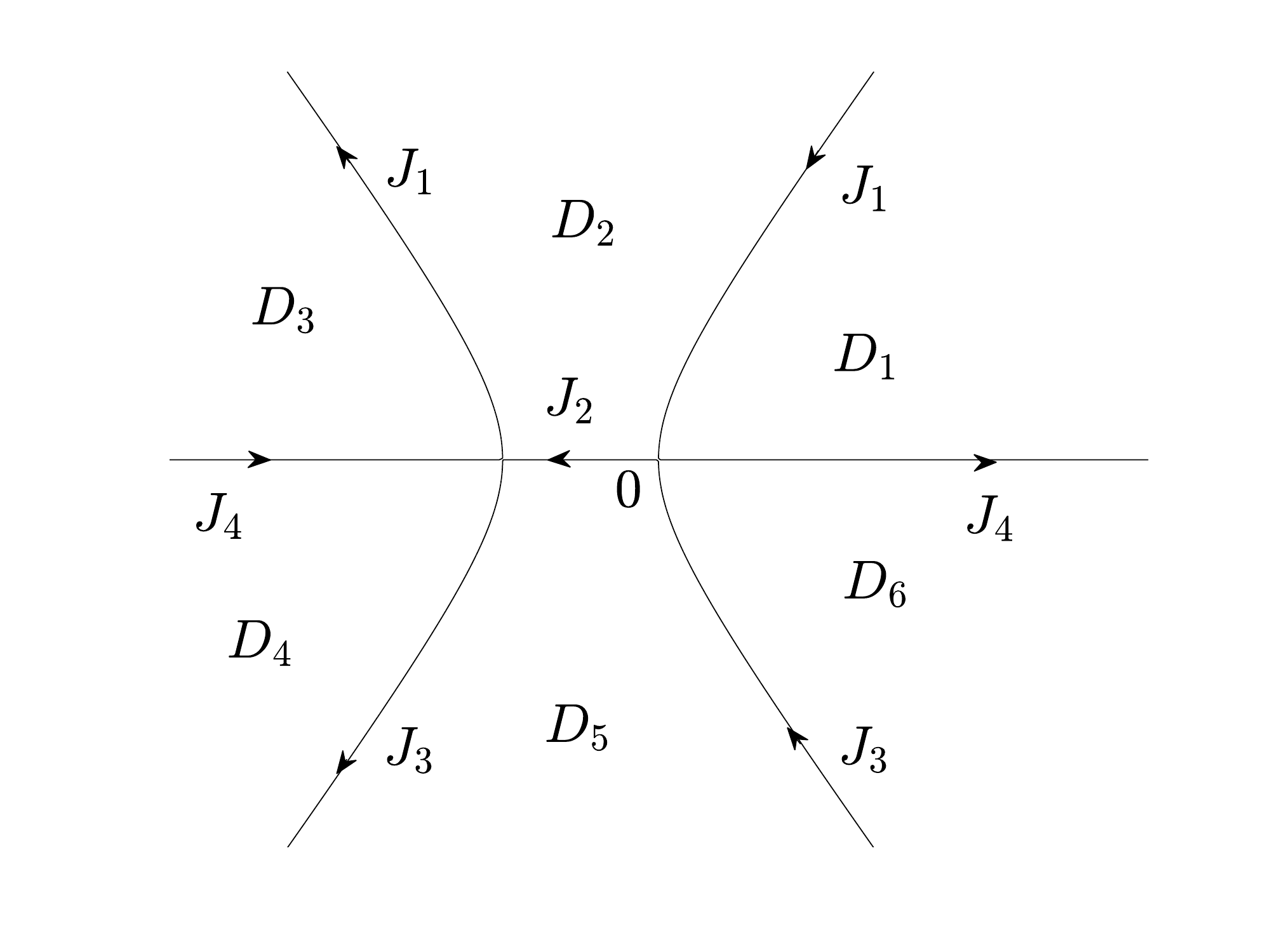}
  \caption{The contour $\Sigma$ for the RH problem.}\label{fig2}
\end{figure}
\bea\label{2.16}
&&J_1(x,t,k)=\begin{pmatrix}
1 & 0\\[4pt]
-h(k)\e^{t\Phi(k)} & 1
\end{pmatrix},\nn\\
&&J_4(x,t,k)=\begin{pmatrix}
1+r_1(k)\overline{r_1(\bar{k})} ~& \overline{r_1(\bar{k})}\e^{-t\Phi(k)} \\[4pt]
r_1(k)\e^{t\Phi(k)} ~& 1 \\
\end{pmatrix},\\
&&J_3(x,t,k)=\begin{pmatrix}
1 &~ -\overline{h(\bar{k})}\e^{-t\Phi(k)}\\[4pt]
0 & 1
\end{pmatrix},\nn\\
&&J_2(x,t,k)=(J_1J_4^{-1}J_3)(x,t,k)=\begin{pmatrix}
1 &~ -\overline{r(\bar{k})}\e^{-t\Phi(k)}\\[4pt]
-r(k)\e^{t\Phi(k)} ~& 1+r(k)\overline{r(\bar{k})} \\
\end{pmatrix},\nn
\eea
where
\be\label{2.17}
\Phi(k)=2\ii\big(k\frac{x}{t}+4\beta k^3+2\alpha k^2\big),
\ee
together with the following asymptotics:
\be\label{2.18}
M(x,t;k)=I+O\bigg(\frac{1}{k}\bigg),~k\rightarrow\infty,
\ee
has a unique solution for $(x,t)\in\Omega$ and the limit $\lim_{k\rightarrow\infty}(k M(x,t;k))_{12}$
exists for each $(x,t)\in[0,\infty)\times[0,\infty).$ Moreover, $u(x,t)$ is defined by
\bea\label{2.19}
u(x,t)=2\ii\lim_{k\rightarrow\infty}(k M(x,t;k))_{12}
\eea
satisfies the Hirota equation \eqref{1.1}. Furthermore, $u(x,t)$ satisfies the initial and boundary values conditions
\berr
u(x,0)=u_0(x),~u(0,t)=g_0(t),~u_x(0,t)=g_1(t),~u_{xx}(0,t)=g_2(t).
\eerr
\subsection{A model RH problem}
After the formulation of the main RH problem, the main idea of analysis the long-time behavior is to reduce the original RH problem to a model RH problem which can be solved exactly. The following theorem is turned out suitable for determining asymptotics of a class of RH problems which arise in the study of long-time asymptotics.
\begin{figure}[htbp]
  \centering
  \includegraphics[width=3in]{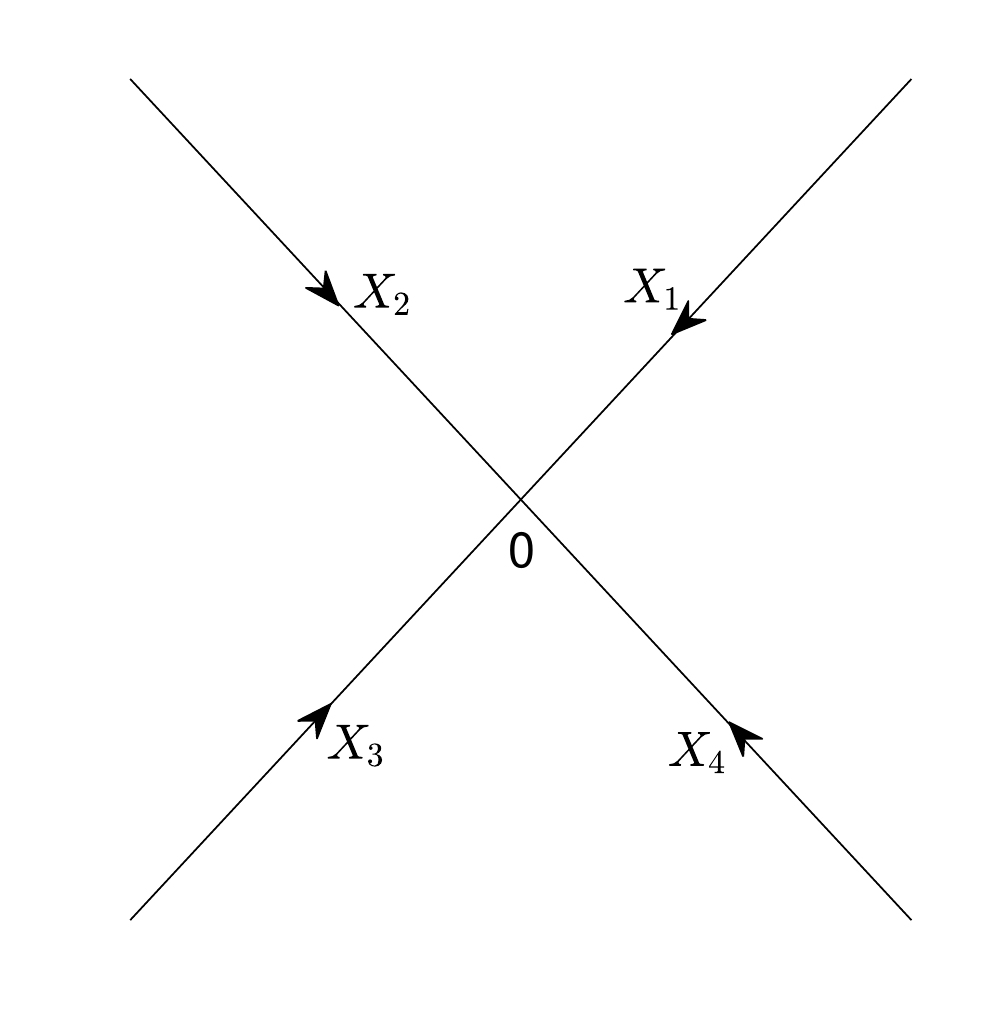}
  \caption{The contour $X=X_1\cup X_2\cup X_3\cup X_4$.}\label{fig5}
\end{figure}

Let $X=X_1\cup X_2\cup X_3\cup X_4\subset\bfC$ be the cross defined by
\be\label{3.40}
\begin{aligned}
X_1&=\{l\e^{\frac{\ii\pi}{4}}|0\leq l<\infty\},~~~ X_2=\{l\e^{\frac{3\ii\pi}{4}}|0\leq l<\infty\},\\
X_3&=\{l\e^{-\frac{3\ii\pi}{4}}|0\leq l<\infty\},~ X_4=\{l\e^{-\frac{\ii\pi}{4}}|0\leq l<\infty\},
\end{aligned}
\ee
and oriented as in Fig. \ref{fig5}. Define the function $\nu:\bfC\rightarrow(0,\infty)$ by $\nu(q)=\frac{1}{2\pi}\ln(1+|q|^2)$. We consider the following RH problems parametrized by $q\in\bfC$:
\be\label{3.42}\left\{
\begin{aligned}
&M^X_+(q,z)=M^X_-(q,z)J^X(q,z),~\text{for~almost~every}~z\in X,\\
&M^X(q,z)\rightarrow I,~~~~\qquad\qquad\qquad\text{as}~z\rightarrow\infty,
\end{aligned}
\right.
\ee
where the jump matrix $J^X(q,z)$ is defined by
\be\label{2.22}
J^X(q,z)=\left\{
\begin{aligned}
&\begin{pmatrix}
1 ~& 0\\[4pt]
-q\e^{\frac{\ii z^2}{2}}z^{2\ii\nu(q)} ~& 1
\end{pmatrix},~~\quad\quad z\in X_1,\\
&\begin{pmatrix}
1 ~& \frac{\bar{q}}{1+|q|^2}\e^{-\frac{\ii z^2}{2}}z^{-2\ii\nu(q)}\\[4pt]
0 ~& 1
\end{pmatrix},~~ z\in X_2,\\
&\begin{pmatrix}
1 ~& 0\\[4pt]
\frac{q}{1+|q|^2}\e^{\frac{\ii z^2}{2}}z^{2\ii\nu(q)} ~& 1
\end{pmatrix},~~\quad z\in X_3,\\
&\begin{pmatrix}
1 ~& -\bar{q}\e^{-\frac{\ii z^2}{2}}z^{-2\ii\nu(q)}\\[4pt]
0 ~& 1
\end{pmatrix},~\quad~ z\in X_4.
\end{aligned}
\right.
\ee
Then the RH problem \eqref{3.42} can be solved explicitly in terms of parabolic cylinder functions \cite{PD,PD1}.
\begin{theorem}\label{th4.1}
The RH problem \eqref{3.42} has a unique solution $M^X(q,z)$ for each $q\in \bfC$. This solution satisfies
\be\label{3.43}
M^X(q,z)=I-\frac{\ii}{z}\begin{pmatrix}
0 ~& \beta^X(q)\\[4pt]
\overline{\beta^X(q)} ~& 0
\end{pmatrix}+O\bigg(\frac{q}{z^2}\bigg),\quad z\rightarrow\infty,~q\in\bfC,
\ee
where the error term is uniform with respect to $\arg z\in[0,2\pi]$ and the function $\beta^X(q)$ is given by
\be\label{3.44}
\beta^X(q)=\sqrt{\nu(q)}\e^{\ii\big(\frac{\pi}{4}-\arg q-\arg\Gamma(\ii\nu(q))\big)},\quad q\in\bfC,
\ee
where $\Gamma(\cdot)$ denotes the standard Gamma function. Moreover, for each compact subset $\mathcal{D}$ of $\bfC$,
\be\label{3.45}
\sup_{q\in\mathcal{D}}\sup_{z\in\bfC\setminus X}|M^X(q,z)|<\infty
\ee
and
\be\label{3.46}
\sup_{q\in\mathcal{D}}\sup_{z\in\bfC\setminus X}\frac{|M^X(q,z)-I|}{|q|}<\infty.
\ee
\end{theorem}
\begin{proof}
The proof of this theorem can be analogously derived by the procedure used in \cite{PD,PD1,JL3}.
\end{proof}
\section{Transformations of the original RH problem}
\setcounter{equation}{0}
\setcounter{lemma}{0}
\setcounter{theorem}{0}
In this section, we aim to transform the associated original RH problem \eqref{2.15} to a solvable RH problem. In proceeding the following analysis, we need the following properties for the functions $r_1(k),~h(k),~r(k)$:

$\bullet$ $r_1(k)$ is smooth and bounded on $\bfR$;

$\bullet$ $h(k)$ is smooth and bounded on $\bar{D}_2$ and analytic in $D_2$;

$\bullet$ $r(k)$ is smooth and bounded on $\bfR$;

$\bullet$ There exist complex constants $\{r_{1,j}\}_{j=1}^\infty$ and $\{h_j\}_{j=1}^\infty$ such that, for any $N\geq1$,
\bea
r_1(k)&=&\sum_{j=1}^N\frac{r_{1,j}}{k^j}+O\bigg(\frac{1}{k^{N+1}}\bigg),~~|k|\rightarrow\infty,~~k\in\bfR,\label{3.1}\\
h(k)&=&\sum_{j=1}^N\frac{h_j}{k^j}+O\bigg(\frac{1}{k^{N+1}}\bigg),~~k\rightarrow\infty,~~k\in\bar{D}_2.\label{3.2}
\eea

By performing a number of transformations, we can bring the RH problem \eqref{2.15} to a form suitable for determining the long-time asymptotics. Let $\xi=\frac{x}{t}$.  The jump matrix $J$ defined in \eqref{2.16} involves the exponentials $\e^{\pm t\Phi}$, where $\Phi(k)$ is defined by
$$\Phi(k)=2\ii\big(k\xi+4\beta k^3+2\alpha k^2\big),\quad k\in\bfC.$$ In particular, we suppose $$\xi<\frac{\alpha^2}{3\beta}.$$
It follows that there are two different real stationary points located at the points where $\frac{\partial\Phi}{\partial k}=0$,
namely, at
\bea
k_1&=&\frac{-\alpha-\sqrt{\alpha^2-3\beta\xi}}{6\beta},\label{3.3}\\
k_2&=&\frac{-\alpha+\sqrt{\alpha^2-3\beta\xi}}{6\beta}\label{3.4}.
\eea

\begin{figure}[htbp]
  \centering
  \includegraphics[width=4in]{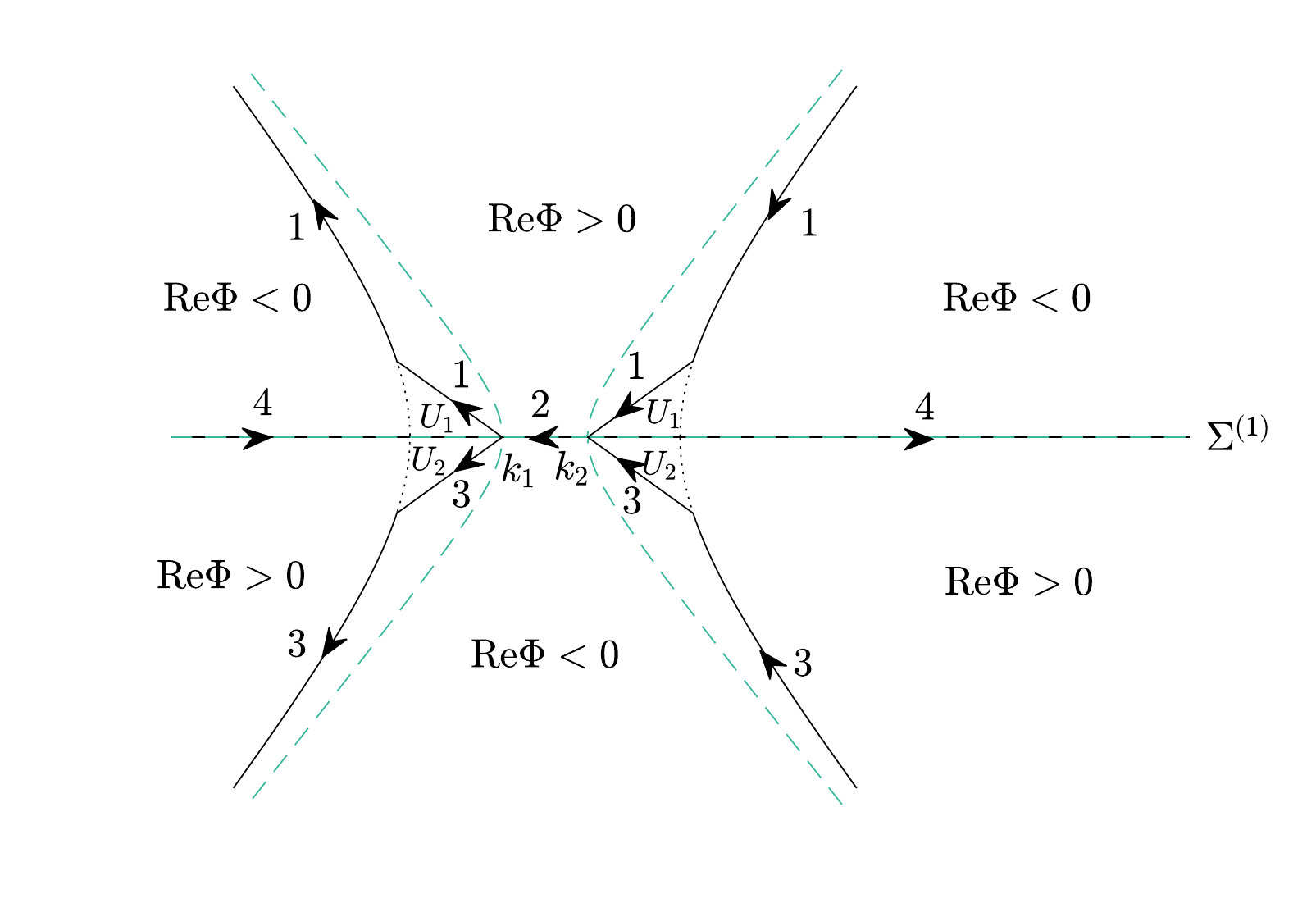}
  \caption{The contour $\Sigma^{(1)}$ and the signature table for Re$\Phi$ (dashed line) in the complex $k$-plane.}\label{fig3}
\end{figure}

The first transformation is to deform the upper/lower plane parts of $\Sigma$ so that it passes through the critical points $k_1$ and $k_2$. Letting $U_1$ and $U_2$ denote the domains shown in Fig. \ref{fig3}. The first transform is:
\be\label{3.5}
M^{(1)}(x,t;k)=M(x,t;k)\times\left\{
\begin{aligned}
&\begin{pmatrix}
1 & 0\\[4pt]
-h(k)\e^{t\Phi(k)} & 1
\end{pmatrix},~~~~k\in U_1,\\
&\begin{pmatrix}
1 &~ \overline{h(\bar{k})}\e^{-t\Phi(k)}\\[4pt]
0 & 1
\end{pmatrix},~~~~k\in U_2,\\
&I,\qquad\qquad\qquad\qquad~~ \text{elsewhere}.
\end{aligned}
\right.
\ee
Then we obtain the RH problem
\be\label{3.6}
M^{(1)}_+(x,t;k)=M^{(1)}_-(x,t;k)J^{(1)}(x,t,k)
\ee
on the contour $\Sigma^{(1)}$ depicted in Fig. \ref{fig3}. The jump matrix $J^{(1)}(x,t,k)$ is given by
\bea
J^{(1)}_1&=&\begin{pmatrix}
1 & 0\\[4pt]
-h\e^{t\Phi} & 1
\end{pmatrix},\qquad J^{(1)}_2=\begin{pmatrix}
1 &~ -\bar{r}\e^{-t\Phi}\\[4pt]
-r\e^{t\Phi} ~& 1+r\bar{r} \\
\end{pmatrix},\nn\\
J^{(1)}_3&=&\begin{pmatrix}
1 &~ -\bar{h}\e^{-t\Phi}\\[4pt]
0 & 1
\end{pmatrix},~~J^{(1)}_4=\begin{pmatrix}
1+r_1\bar{r}_1 ~& \bar{r}_1\e^{-t\Phi} \\[4pt]
r_1\e^{t\Phi} ~& 1 \\
\end{pmatrix},\nn
\eea
where $J^{(1)}_i$ denotes the restriction of $J^{(1)}$ to the contour labeled by $i$ in Fig. \ref{fig3}.

The next transformation is:
\be\label{3.7}
M^{(2)}(x,t;k)=M^{(1)}(x,t;k)\delta^{-\sigma_3}(k),
\ee
where the complex-valued function $\delta(k)$ is given by
\be\label{3.8}
\delta(k)=\exp\bigg\{\frac{1}{2\pi\ii}\int^{k_2}_{k_1}\frac{\ln(1+|r(s)|^2)}{s-k}\dd s\bigg\},\quad k\in\bfC\setminus[k_1,k_2].
\ee
\begin{lemma}
The function $\delta(k)$ has the following properties:

(i) $\delta(k)$ satisfies the following jump condition across the real axis oriented in Fig. \ref{fig3}:
\berr
\delta_+(k)=\left\{
\begin{aligned}
&\frac{\delta_-(k)}{1+|r(k)|^2},~~k\in(k_1,k_2),\\
&\delta_-(k),\quad\qquad k\in\bfR\setminus[k_1,k_2].
\end{aligned}
\right.
\eerr

(ii) As $k\rightarrow\infty$, $\delta(k)$ satisfies the asymptotic formula
\be\label{3.9}
\delta(k)=1+O(k^{-1}),\quad k\rightarrow\infty.
\ee

(iii) $\delta(k)$ and $\delta^{-1}(k)$ are bounded and analytic functions of $k\in\bfC\setminus[k_1,k_2]$ with continuous boundary values on $(k_1,k_2)$.

(iv) $\delta(k)$ obeys the symmetry $$\delta(k)=\overline{\delta(\bar{k})}^{-1},\quad k\in\bfC\setminus[k_1,k_2].$$
\end{lemma}

Then $M^{(2)}(x,t;k)$ satisfies the following RH problem
\be\label{3.10}
M^{(2)}_+(x,t;k)=M^{(2)}_-(x,t;k)J^{(2)}(x,t,k)
\ee
with the contour $\Sigma^{(2)}=\Sigma^{(1)}$ and the jump matrix $J^{(2)}=\delta_-^{\sigma_3}J^{(1)}\delta_+^{-\sigma_3}$, namely,
\bea
J^{(2)}_1&=&\begin{pmatrix}
1 & 0\\[4pt]
-h\delta^{-2}\e^{t\Phi} & 1
\end{pmatrix},\qquad J^{(2)}_2=\begin{pmatrix}
1 ~& -r_2\delta_-^{2}\e^{-t\Phi} \\[4pt]
0 ~& 1 \\
\end{pmatrix}\begin{pmatrix}
1 ~& 0 \\[4pt]
-\bar{r}_2\delta_+^{-2}\e^{t\Phi} ~& 1 \\
\end{pmatrix},\nn\\
J^{(2)}_3&=&\begin{pmatrix}
1 &~ -\bar{h}\delta^{2}\e^{-t\Phi}\\[4pt]
0 & 1
\end{pmatrix}, \quad J^{(2)}_4=\begin{pmatrix}
1 ~& \bar{r}_1\delta^{2}\e^{-t\Phi} \\[4pt]
0 ~& 1 \\
\end{pmatrix}\begin{pmatrix}
1 ~& 0 \\[4pt]
r_1\delta^{-2}\e^{t\Phi} ~& 1 \\
\end{pmatrix},\nn
\eea
where we define $r_2(k)$ by
\be\label{3.11}
r_2(k)=\frac{\overline{r(\bar{k})}}{1+r(k)\overline{r(\bar{k})}}.
\ee

Let $N>1$ be given, and let $\mathcal{I}$ denote the interval $\mathcal{I}=(0,N]\cap(0,\frac{\alpha^2}{3\beta})$.
Before processing the next deformation, we follow the idea of \cite{JL1,JL2} and decompose each of the functions $h$, $r_1$, $r_2$ into an analytic part and a small remainder because the spectral functions have limited domains of analyticity. The analytic part of the jump matrix will be deformed, whereas the small remainder will be left on the original contour. In fact, we have the following lemmas.
\begin{lemma}\label{lem1}
There exist a decomposition
\berr
h(k)=h_a(t,k)+h_r(t,k),\quad t>0,~~k\in(\bar{D}_1\cap\bar{D}_2)\cup(\bar{D}_3\cap\bar{D}_2),
\eerr
where the functions $h_a$ and $h_r$ have the following properties:

(1) For each $t>0$, $h_a(t,k)$ is defined and continuous for $k\in\bar{D}_1\cup\bar{D}_3$ and analytic for $k\in D_1\cup D_3.$

(2) For each $\xi\in\mathcal{I}$ and each $t>0$, the function $h_a(t,k)$ satisfies
\be\label{3.12}
|h_a(t,k)|\leq \frac{C}{1+|k|^2}\e^{\frac{t}{4}|\text{Re}\Phi(k)|},\quad~k\in\bar{D}_1\cup\bar{D}_3,
\ee
where the constant $C$ is independent of $\xi, k, t$.

(3) The $L^1$, $L^2$ and $L^\infty$ norms of the function $h_r(t,\cdot)$ on $(\bar{D}_1\cap\bar{D}_2)\cup(\bar{D}_3\cap\bar{D}_2)$ are $O(t^{-3/2})$ as $t\rightarrow\infty$.
\end{lemma}
\begin{proof}
We first consider the decomposition of $h(k)$ for $k\in\bar{D}_1\cap\bar{D}_2$. Since $h(k)\in C^\infty(\bar{D}_2)$, then for $k\in\bar{D}_1\cap\bar{D}_2,~n=0,1,2$,
\bea
h^{(n)}(k)&=&\frac{\dd^n}{\dd k^n}\bigg(\sum_{j=0}^4\frac{h^{(j)}(0)}{j!}k^j\bigg)+O(k^{5-n}),~~ k\rightarrow0,\label{3.13}\\
h^{(n)}(k)&=&\frac{\dd^n}{\dd k^n}\bigg(\sum_{j=1}^3h_jk^{-j}\bigg)+O(k^{-4-n}),~~~~k\rightarrow\infty.\label{3.14}
\eea
Let
\be\label{3.15}
f_0(k)=\sum_{j=2}^{9}\frac{a_j}{(k+\ii)^j},
\ee
where $\{a_j\}_2^9$ are complex constants such that
\be\label{3.16}
f_0(k)=\left\{
\begin{aligned}
&\sum_{j=0}^4\frac{h^{(j)}(0)}{j!}k^j+O(k^5),~~k\rightarrow0,\\
&\sum_{j=1}^3h_jk^{-j}+O(k^{-4}),~~~~~k\rightarrow \infty.
\end{aligned}
\right.
\ee
It is easy to verify that \eqref{3.16} imposes eight linearly independent conditions on the $a_j$, hence the coefficients $a_j$ exist and are unique. Letting $f=h-f_0$, it follows that \\
(i) $f_0(k)$ is a rational function of $k\in\bfC$ with no poles in $\bar{D}_1$;\\
(ii) $f_0(k)$ coincides with $h(k)$ to four order at 0 and to order three at $\infty$, more precisely,
\be\label{3.17}
\frac{\dd^n}{\dd k^n}f(k)=\left\{
\begin{aligned}
&O(k^{5-n}),~\quad k\rightarrow0,\\
&O(k^{-4-n}),~~ k\rightarrow\infty,
\end{aligned}
\quad k\in\bar{D}_1\cap\bar{D}_2,~n=0,1,2.
\right.
\ee
The decomposition of $h(k)$ can be derived as follows. The map $k\mapsto\psi=\psi(k)$ defined by $\psi(k)=8\beta k^3+4\alpha k^2$ is a bijection $\bar{D}_1\cap\bar{D}_2\mapsto\bfR$, so we may define a function $F$ : $\bfR\rightarrow\bfC$ by
\be\label{3.18}
F(\psi)=(k+\ii)^2f(k),~\psi\in\bfR\setminus\{0\},
\ee
$F(\psi)$ is $C^5$ for $\psi\neq0$ and
\berr
F^{(n)}(\psi)=\bigg(\frac{1}{24\beta k(k-k_0)}\frac{\partial}{\partial k}\bigg)^n\bigg((k+\ii)^2f(k)\bigg),~~\psi\in\bfR\setminus\{0\}.
\eerr
By \eqref{3.17}, $F\in C^1(\bfR)$ and $F^{(n)}(\psi)=O(|\psi|^{-2/3})$ as $|\psi|\rightarrow\infty$ for $n=0,1,2$. In particular,
\be\label{3.19}
\bigg\|\frac{\dd^nF}{\dd\psi^n}\bigg\|_{L^2(\bfR)}<\infty,\quad n=0,1,2,
\ee
that is, $F$ belongs to $H^2(\bfR)$. By the Fourier transform $\hat{F}(s)$ defined by
\be\label{3.20}
\hat{F}(s)=\frac{1}{2\pi}\int_\bfR F(\psi)\e^{-\ii\psi s}\dd\psi
\ee
where
\be\label{3.21}
F(\psi)=\int_\bfR\hat{F}(s)\e^{\ii\psi s}\dd s,
\ee
it follows from Plancherel theorem that $\|s^2\hat{F}(s)\|_{L^2(\bfR)}<\infty$. Equations \eqref{3.18} and \eqref{3.21} imply
\be
f(k)=\frac{1}{(k+\ii)^2}\int_\bfR\hat{F}(s)\e^{\ii\psi s}\dd s,\quad k\in\bar{D}_1\cap\bar{D}_2.
\ee
Writing
$$f(k)=f_a(t,k)+f_r(t,k),\quad t>0,~k\in\bar{D}_1\cap\bar{D}_2,$$
where the functions $f_a$ and $f_r$ are defined by
\bea
f_a(t,k)&=&\frac{1}{(k+\ii)^2}\int_{-\frac{t}{4}}^\infty\hat{F}(s)\e^{\ii(8\beta k^3+4\alpha k^2)s}\dd s,\quad t>0,~k\in\bar{D}_1,\\
f_r(t,k)&=&\frac{1}{(k+\ii)^2}\int^{-\frac{t}{4}}_{-\infty}\hat{F}(s)\e^{\ii(8\beta k^3+4\alpha k^2)s}\dd s,\quad t>0,~k\in\bar{D}_1\cap\bar{D}_2,
\eea
we infer that $f_a(t,\cdot)$ is continuous in $\bar{D}_1$ and analytic in $D_1$. Moreover, since $|\text{Re}[\ii(8\beta k^3+4\alpha k^2)]|\leq|\text{Re}\Phi(k)|$ for $k\in\bar{D}_1$ and $\xi\in\mathcal{I}$, we can get
\bea\label{3.25}
|f_a(t,k)|&\leq&\frac{C}{|k+\ii|^2}\|\hat{F}(s)\|_{L^1(\bfR)}\sup_{s\geq-\frac{t}{4}}\e^{s\text{Re}[\ii(8\beta k^3+4\alpha k^2)]}\nn\\
&\leq&\frac{C}{1+|k|^2}\e^{\frac{t}{4}|\text{Re}\Phi(k)|},\quad t>0,~k\in\bar{D}_1,~\xi\in\mathcal{I}.
\eea
Furthermore, we have
\bea\label{3.26}
|f_r(t,k)|&\leq&\frac{C}{|k+\ii|^2}\int_{-\infty}^{-\frac{t}{4}}s^2|\hat{F}(s)|s^{-2}\dd s\nn\\
&\leq&\frac{C}{1+|k|^2}\|s^2\hat{F}(s)\|_{L^2(\bfR)}\sqrt{\int_{-\infty}^{-\frac{t}{4}}s^{-4}\dd s},\\
&\leq&\frac{C}{1+|k|^2}t^{-3/2},\quad t>0,~k\in\bar{D}_1\cap\bar{D}_2,~\xi\in\mathcal{I}.\nn
\eea
Hence, the $L^1,L^2$ and $L^\infty$ norms of $f_r$ on $\bar{D}_1\cap\bar{D}_2$ are $O(t^{-3/2})$. Letting
\bea
h_a(t,k)&=&f_0(k)+f_a(t,k),\quad t>0,~ k\in\bar{D}_1,\\
h_r(t,k)&=&f_r(t,k),~~\qquad\qquad t>0,~ k\in\bar{D}_1\cap\bar{D}_2.
\eea
For $k\in\bar{D}_3\cap\bar{D}_2$, a similar decomposition of $h$ can also be obtained.  Thus, we find a decomposition of $h$ for $k\in(\bar{D}_1\cap\bar{D}_2)\cup(\bar{D}_3\cap\bar{D}_2)$ with the properties listed in the statement of the lemma.
\end{proof}

We introduce the open subsets $\{\Omega_j\}_1^8$, as displayed in Fig. \ref{fig4}. The following lemma describes how to decompose $r_j$, $j=1,2$ into an analytic part $r_{j,a}$ and a small remainder $r_{j,r}$. A proof can be found in \cite{JL3}.

\begin{figure}[htbp]
  \centering
  \includegraphics[width=6in]{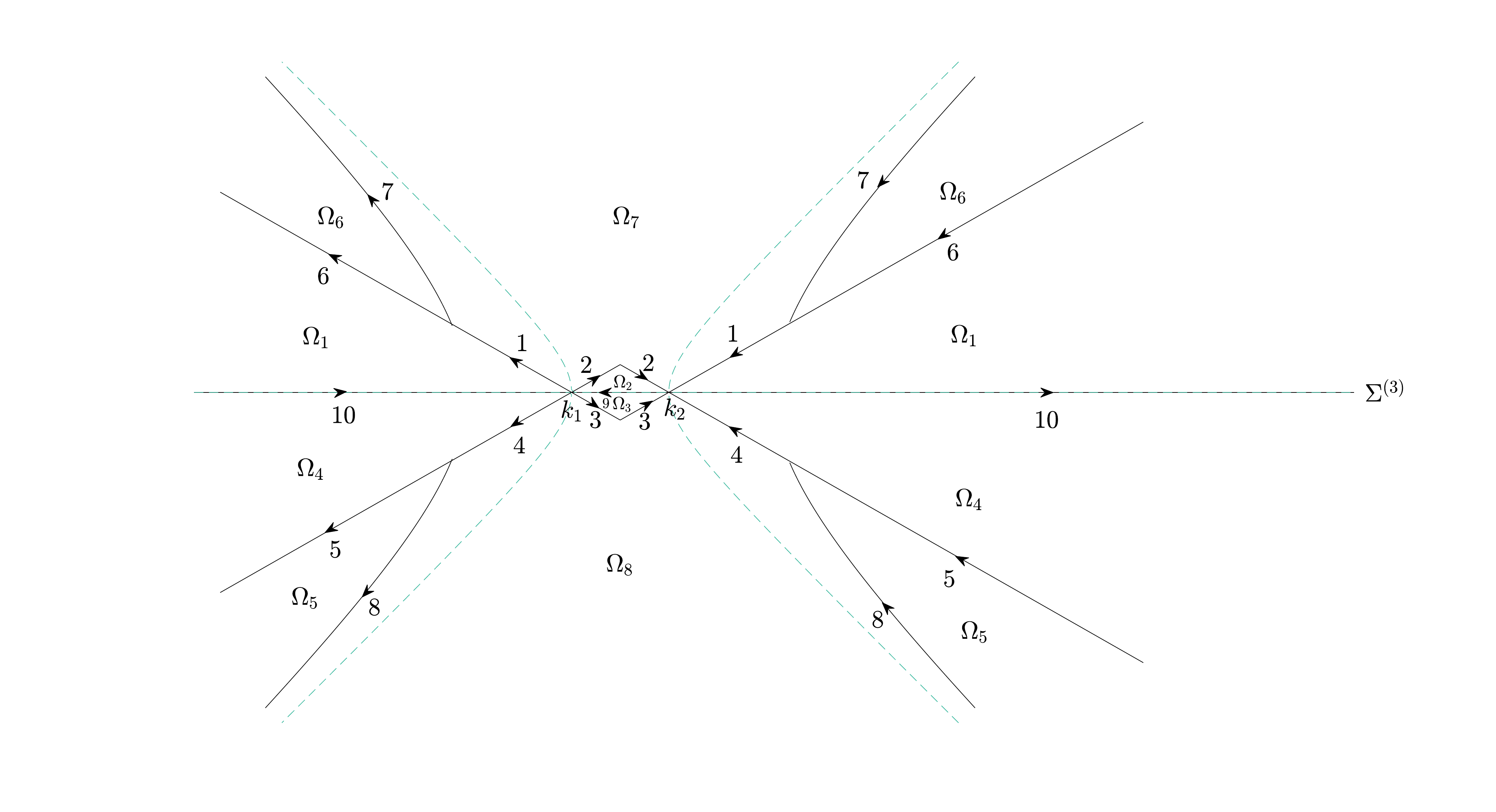}
  \caption{The contour $\Sigma^{(3)}$ and the open sets $\{\Omega_j\}_1^8$ in the complex $k$-plane.}\label{fig4}
\end{figure}

\begin{lemma}\label{lem2}
There exist decompositions
\be
\begin{aligned}
r_1(k)&=r_{1,a}(x,t,k)+r_{1,r}(x,t,k),\quad k\in(-\infty,k_1)\cup(k_2,\infty),\\
r_2(k)&=r_{2,a}(x,t,k)+r_{2,r}(x,t,k),\quad k\in(k_1,k_2),
\end{aligned}
\ee
where the functions $\{r_{j,a}, r_{j,r}\}^2_1$ have the following properties:

(1) For each $\xi\in\mathcal{I}$ and each $t>0$, $r_{j,a}(x,t,k)$ is defined and continuous for $k\in\bar{\Omega}_j$ and
analytic for $\Omega_j$, $j= 1,2$.

(2) The functions $r_{1,a}$ and $r_{2,a}$ satisfy
\be\label{3.30}
|r_{j,a}(x,t,k)|\leq \frac{C}{1+|k|^2}\e^{\frac{t}{4}|\text{Re}\Phi(k)|},
\quad t>0,~k\in\bar{\Omega}_j,~\xi\in\mathcal{I},~j=1,2,
\ee
where the constant $C$ is independent of $\xi, k, t$.

(3) The $L^1, L^2$ and $L^\infty$ norms of the function $r_{1,r}(x,t,\cdot)$ on $(-\infty,k_1)\cup(k_2,\infty)$ are $O(t^{-3/2})$ as $t\rightarrow\infty$ uniformly with respect to $\xi\in\mathcal{I}$.

(4) The $L^1, L^2$ and $L^\infty$ norms of the function $r_{2,r}(x,t,\cdot)$ on $(k_1,k_2)$ are $O(t^{-3/2})$ as $t\rightarrow\infty$ uniformly with respect to $\xi\in\mathcal{I}$.
\end{lemma}

The purpose of the next transformation is to deform the contour so that the jump matrix involves the exponential factor $\e^{-t\Phi}$ on the parts of the contour where Re$\Phi$ is positive and the factor $\e^{t\Phi}$ on the parts where Re$\Phi$ is negative. More precisely, we put
\be\label{3.31}
M^{(3)}(x,t;k)=M^{(2)}(x,t;k)G(k),
\ee
where
\be\label{3.32}
G(k)=\left\{
\begin{aligned}
&\begin{pmatrix}
1 &~ 0\\[4pt]
-r_{1,a}\delta^{-2}\e^{t\Phi} ~& 1
\end{pmatrix},~~\quad k\in\Omega_1,\\
&\begin{pmatrix}
1 ~& -r_{2,a}\delta^{2}\e^{-t\Phi}\\[4pt]
0 ~& 1
\end{pmatrix},\qquad k\in\Omega_2,\\
&\begin{pmatrix}
1 ~& 0\\[4pt]
\bar{r}_{2,a}\delta^{-2}\e^{t\Phi} ~& 1
\end{pmatrix},\qquad~~ k\in\Omega_3,\\
&\begin{pmatrix}
1 ~& \bar{r}_{1,a}\delta^2\e^{-t\Phi}\\[4pt]
0 ~& 1
\end{pmatrix},~~\qquad k\in\Omega_4,\\
&\begin{pmatrix}
1 ~& -\bar{h}_{a}\delta^{2}\e^{-t\Phi} \\[4pt]
0 ~& 1
\end{pmatrix},~\qquad k\in\Omega_5,\\
&\begin{pmatrix}
1 ~& 0\\[4pt]
h_{a}\delta^{-2}\e^{t\Phi} ~& 1
\end{pmatrix},~~~\qquad k\in\Omega_6,\\
&I,~~~~~~~\qquad\qquad\qquad\quad k\in\Omega_7\cup\Omega_8.
\end{aligned}
\right.
\ee
Then the matrix $M^{(3)}(x,t;k)$ satisfies the following RH problem
\be\label{3.33}
M_+^{(3)}(x,t;k)=M_-^{(3)}(x,t;k)J^{(3)}(x,t,k),
\ee
with the jump matrix $J^{(3)}=G_+^{-1}(k)J^{(2)}G_-(k)$ is given by
\bea
J^{(3)}_1&=&\begin{pmatrix}
1 & 0\\[4pt]
-(r_{1,a}+h)\delta^{-2}\e^{t\Phi} & 1
\end{pmatrix},\quad J^{(3)}_2=\begin{pmatrix}
1 &~ r_{2,a}\delta^2\e^{-t\Phi}\\
0~& 1
\end{pmatrix},\quad J^{(3)}_3=\begin{pmatrix}
1 &~ 0\\
\bar{r}_{2,a}\delta^{-2}\e^{t\Phi} ~& 1
\end{pmatrix},\nn\\
J^{(3)}_4&=&\begin{pmatrix}
1 &~ -(\bar{r}_{1,a}+\bar{h})\delta^2\e^{-t\Phi}\\
0~& 1
\end{pmatrix},\quad~
J^{(3)}_5=\begin{pmatrix}
1 &~ -(\bar{r}_{1,a}+\bar{h}_a)\delta^2\e^{-t\Phi}\\
0~& 1
\end{pmatrix},\nn\\
J^{(3)}_6&=&\begin{pmatrix}
1 &~ 0\\
-(r_{1,a}+h_a)\delta^{-2}\e^{t\Phi}~& 1
\end{pmatrix},\quad
J^{(3)}_7=\begin{pmatrix}
1 &~ 0\\[4pt]
-h_r\delta^{-2}\e^{t\Phi} & 1
\end{pmatrix},\quad
J^{(3)}_8=\begin{pmatrix}
1 &~ -\bar{h}_r\delta^{2}\e^{-t\Phi}\\[4pt]
0 & 1
\end{pmatrix},\nn\\
J^{(3)}_9&=&\begin{pmatrix}
1 ~& -r_{2,r}\delta_-^{2}\e^{-t\Phi} \\[4pt]
0 ~& 1 \\
\end{pmatrix}\begin{pmatrix}
1 ~& 0 \\[4pt]
-\bar{r}_{2,r}\delta_+^{-2}\e^{t\Phi} ~& 1 \\
\end{pmatrix}, \quad J^{(3)}_{10}=\begin{pmatrix}
1 ~& \bar{r}_{1,r}\delta^{2}\e^{-t\Phi} \\[4pt]
0 ~& 1 \\
\end{pmatrix}\begin{pmatrix}
1 ~& 0 \\[4pt]
r_{1,r}\delta^{-2}\e^{t\Phi} ~& 1 \\
\end{pmatrix},\nn
\eea
with $J^{(3)}_i$ denoting the restriction of $J^{(3)}$ to the contour labeled by $i$ in Fig. \ref{fig4}.

Obviously, the jump matrix $J^{(3)}$ decays to identity matrix $I$ as $t\rightarrow\infty$ everywhere except near the critical points $k_1$ and $k_2$. This implies that we only need to consider a neighborhood of the critical points $k_1$ and $k_2$ when we studying the long-time asymptotics of $M^{(3)}(x,t;k)$ in terms of the corresponding RH problem.

\section{Local models near $k_1$ and $k_2$}
\setcounter{equation}{0}
\setcounter{lemma}{0}
\setcounter{theorem}{0}

To focus on $k_1$ and $k_2$. We introduce the following scaling operators
\bea
&&S_{k_1}:~k\mapsto\frac{z}{\sqrt{-8t(\alpha+6\beta k_1)}}+k_1,\\
&&S_{k_2}:~k\mapsto\frac{z}{\sqrt{8t(\alpha+6\beta k_2)}}+k_2.
\eea
For $j=1,2$, let $D_\varepsilon(k_j)$ denote the open disk of radius $\varepsilon$ centered at $k_j$ for a small $\varepsilon>0$. Then, the map $k\mapsto z$ is a bijection from $D_\varepsilon(k_j)$ to the open disk of radius $\sqrt{-8t(\alpha+6\beta k_1)}\varepsilon$ and $\sqrt{8t(\alpha+6\beta k_2)}\varepsilon$ centered at the origin for all $\xi\in\mathcal{I}$, respectively. Integrating by parts in formula \eqref{3.8} yields,
\be\label{3.34}
\delta(k)=\bigg(\frac{k-k_2}{k-k_1}\bigg)^{-\ii\nu(k_1)}\e^{\chi_1(k)}=\bigg(\frac{k-k_2}{k-k_1}\bigg)^{-\ii\nu(k_2)}\e^{\chi_2(k)},
\ee
where
\bea
\nu(k_1)&=&\frac{1}{2\pi}\ln(1+|r(k_1)|^2)>0,\label{3.35}\\
\chi_1(k)&=&\frac{1}{2\pi\ii}\int_{k_1}^{k_2}\ln\bigg(\frac{1+|r(s)|^2}{1+|r(k_1)|^2}\bigg)\frac{\dd s}{s-k},\label{3.36}\\
\nu(k_2)&=&\frac{1}{2\pi}\ln(1+|r(k_2)|^2)>0,\label{3.35'}\\
\chi_2(k)&=&\frac{1}{2\pi\ii}\int_{k_1}^{k_2}\ln\bigg(\frac{1+|r(s)|^2}{1+|r(k_2)|^2}\bigg)\frac{\dd s}{s-k}.\label{3.36'}
\eea
Hence, we have
\bea
S_{k_1}(\delta(k)\e^{-\frac{t\Phi(k)}{2}})&=&\delta_{k_1}^0(z)\delta_{k_1}^1(z),\nn\\
S_{k_2}(\delta(k)\e^{-\frac{t\Phi(k)}{2}})&=&\delta_{k_2}^0(z)\delta_{k_2}^1(z),\nn
\eea
with
\bea
\delta_{k_1}^0(z)&=&\bigg(-8tk_1^2(\alpha+6\beta k_1)\bigg)^{-\frac{\ii\nu(k_1)}{2}}\e^{\chi_1(k_1)}\e^{2\ii k_1^2t(\alpha+4\beta k_1)},\label{0.1}\\
\delta_{k_1}^1(z)&=&(-z)^{\ii\nu(k_1)}\exp\bigg(\frac{\ii z^2}{4}\bigg(1-\frac{\beta z}{\sqrt{2t}(-(\alpha+6\beta k_1))^{3/2}}\bigg)\bigg)\nn\\
&&\bigg(\frac{-k_1}{-z/\sqrt{-8t(\alpha+6\beta k_1)}+k_2-k_1}\bigg)^{\ii\nu(k_1)}\e^{(\chi_1([z/\sqrt{-8t(\alpha+6\beta k_1)}]+k_1)-\chi_1(k_1))},\label{0.2}\\
\delta_{k_2}^0(z)&=&\bigg(8tk_1^2(\alpha+6\beta k_2)\bigg)^{\frac{\ii\nu(k_2)}{2}}\e^{\chi_2(k_2)}\e^{2\ii k_2^2t(\alpha+4\beta k_2)},\label{3.37}\\
\delta_{k_2}^1(z)&=&z^{-\ii\nu(k_2)}\exp\bigg(-\frac{\ii z^2}{4}\bigg(1+\frac{\beta z}{\sqrt{2t}(\alpha+6\beta k_2)^{3/2}}\bigg)\bigg)\nn\\
&&\bigg(\frac{-k_1}{z/\sqrt{8t(\alpha+6\beta k_2)}+k_2-k_1}\bigg)^{-\ii\nu(k_2)}\e^{(\chi_2([z/\sqrt{8t(\alpha+6\beta k_2)}]+k_2)-\chi_2(k_2))}.\label{3.38}
\eea
Now we define
\bea
\check{M}(x,t;z)&=&M^{(3)}(x,t;k)(\delta_{k_1}^0)^{\sigma_3}(z),\quad k\in D_\varepsilon(k_1)\setminus\Sigma^{(3)},\label{3.39}\\
\tilde{M}(x,t;z)&=&M^{(3)}(x,t;k)(\delta_{k_2}^0)^{\sigma_3}(z),\quad k\in D_\varepsilon(k_2)\setminus\Sigma^{(3)}.
\eea
Then $\check{M}$ and $\tilde{M}$ are the sectionally analytic functions of $z$ which satisfies
\bea
\check{M}_+(x,t;z)&=&\check{M}_-(x,t;z)\check{J}(x,t,z),\quad k\in \mathcal{X}_{k_1}^\varepsilon,\nn\\
\tilde{M}_+(x,t;z)&=&\tilde{M}_-(x,t;z)\tilde{J}(x,t,z),\quad k\in \mathcal{X}_{k_2}^\varepsilon,\nn
\eea
where $\mathcal{X}_{k_j}=X+k_j$ denote the cross $X$ defined by \eqref{3.40} centered at $k_j$ and $\mathcal{X}_{k_j}^\varepsilon=\mathcal{X}_{k_j}\cap D_\varepsilon(k_j)$ for $j=1,2$.

The corresponding jump matrices
\bea
\check{J}(x,t,z)&=&(\delta_{k_1}^0)^{-\hat{\sigma}_3}(z)J^{(3)}(x,t,k),\label{0.5}\\
\tilde{J}(x,t,z)&=&(\delta_{k_2}^0)^{-\hat{\sigma}_3}(z)J^{(3)}(x,t,k)\label{3.41}
\eea
are given by
\berr
\check{J}(x,t,z)=\left\{
\begin{aligned}
&\begin{pmatrix}
1 ~& -r_{2,a}(\delta_{k_1}^1)^{2}\\[4pt]
0 ~& 1
\end{pmatrix},\qquad\qquad k\in (\mathcal{X}_{k_1}^\varepsilon)_1,\\
&\begin{pmatrix}
1 ~& 0\\[4pt]
(r_{1,a}+h_a)(\delta_{k_1}^1)^{-2} ~& 1
\end{pmatrix},\quad k\in (\mathcal{X}_{k_1}^\varepsilon)_2\cap D_3,\\
&\begin{pmatrix}
1 ~& 0\\[4pt]
(r_{1,a}+h)(\delta_{k_1}^1)^{-2} ~& 1
\end{pmatrix},\quad~ k\in (\mathcal{X}_{k_1}^\varepsilon)_2\cap D_2,\\
&\begin{pmatrix}
1 ~& (\bar{r}_{1,a}+\bar{h})(\delta_{k_1}^1)^{2}\\[4pt]
0 ~& 1
\end{pmatrix},\quad\quad k\in (\mathcal{X}_{k_1}^\varepsilon)_3\cap D_5,\\
&\begin{pmatrix}
1 ~& (\bar{r}_{1,a}+\bar{h}_a)(\delta_{k_1}^1)^{2}\\[4pt]
0 ~& 1
\end{pmatrix},~~\quad k\in (\mathcal{X}_{k_1}^\varepsilon)_3\cap D_4,\\
&\begin{pmatrix}
1 ~& 0\\[4pt]
-\bar{r}_{2,a}(\delta_{k_1}^1)^{-2} ~& 1
\end{pmatrix},~~\qquad~~ k\in (\mathcal{X}_{k_1}^\varepsilon)_4,\\
\end{aligned}
\right.
\eerr
and
\berr
\tilde{J}(x,t,z)=\left\{
\begin{aligned}
&\begin{pmatrix}
1 ~& 0\\[4pt]
-(r_{1,a}+h_a)(\delta_{k_2}^1)^{-2} ~& 1
\end{pmatrix},\quad k\in (\mathcal{X}_{k_2}^\varepsilon)_1\cap D_1,\\
&\begin{pmatrix}
1 ~& 0\\[4pt]
-(r_{1,a}+h)(\delta_{k_2}^1)^{-2} ~& 1
\end{pmatrix},\quad~ k\in (\mathcal{X}_{k_2}^\varepsilon)_1\cap D_2,\\
&\begin{pmatrix}
1 ~& r_{2,a}(\delta_{k_2}^1)^{2}\\[4pt]
0 ~& 1
\end{pmatrix},~~~~\qquad\qquad k\in (\mathcal{X}_{k_2}^\varepsilon)_2,\\
&\begin{pmatrix}
1 ~& 0\\[4pt]
\bar{r}_{2,a}(\delta_{k_2}^1)^{-2} ~& 1
\end{pmatrix},~~\qquad\qquad k\in (\mathcal{X}_{k_2}^\varepsilon)_3,\\
&\begin{pmatrix}
1 ~& -(\bar{r}_{1,a}+\bar{h})(\delta_{k_2}^1)^{2}\\[4pt]
0 ~& 1
\end{pmatrix},\quad\quad k\in (\mathcal{X}_{k_2}^\varepsilon)_4\cap D_5,\\
&\begin{pmatrix}
1 ~& -(\bar{r}_{1,a}+\bar{h}_a)(\delta_{k_2}^1)^{2}\\[4pt]
0 ~& 1
\end{pmatrix},~~\quad k\in (\mathcal{X}_{k_2}^\varepsilon)_4\cap D_6.\\
\end{aligned}
\right.
\eerr
The sets $(\mathcal{X}_{k_1}^\varepsilon)_2\cap D_3$, $(\mathcal{X}_{k_1}^\varepsilon)_3\cap D_4$, $(\mathcal{X}_{k_2}^\varepsilon)_1\cap D_1$ and $(\mathcal{X}_{k_2}^\varepsilon)_4\cap D_6$ are may be empty for sufficient small $\varepsilon$.

For the jump matrix $\tilde{J}(x,t,z)$, define $$q=r(k_2),$$ then for any fixed $z\in X$, we have $k(z)\rightarrow k_2$ as $t\rightarrow\infty$. Hence,
$$r_{1,a}(k)+h(k)\rightarrow q,\quad r_{2,a}(k)\rightarrow\frac{\bar{q}}{1+|q|^2},\quad \delta_{k_2}^1\rightarrow\e^{-\frac{\ii z^2}{4}}z^{-\ii\nu(q)}.$$
This implies that the jump matrix $\tilde{J}$ tend to the matrix $J^X$ defined in \eqref{2.22} for large $t$. In other words, the jumps of $M^{(3)}$ for $k$ near $k_2$ approach those of the function $M^X(\delta_{k_2}^0)^{-\sigma_3}$ as $t\rightarrow\infty$. Therefore, we can approximate $M^{(3)}$ in the neighborhood $D_\varepsilon(k_2)$ of $k_2$ by
\be\label{3.47}
M^{(k_2)}(x,t;k)=(\delta_{k_2}^0)^{\sigma_3}M^X(q,z)(\delta_{k_2}^0)^{-\sigma_3},
\ee
where $M^X(q,z)$ is given by \eqref{3.43}.

For the case of $\check{J}(x,t,k)$, as $t\rightarrow\infty$, we find
\berr
r_{1,a}(k)+h(k)\rightarrow r(k_1),\quad r_{2,a}(k)\rightarrow\frac{\overline{r(k_1)}}{1+|r(k_1)|^2},\quad \delta_{k_1}^1\rightarrow\e^{\frac{\ii z^2}{4}}(-z)^{\ii\nu(k_1)}.
\eerr
This fact implies
\berr
\check{J}(x,t,z)\rightarrow J^{Y}(p,z)=\left\{
\begin{aligned}
&\begin{pmatrix}
1 ~& -\frac{\bar{p}}{1+|p|^2}\e^{\frac{\ii z^2}{2}}(-z)^{2\ii\nu(p)}\\[4pt]
0 ~& 1
\end{pmatrix},~~\quad z\in X_1,\\
&\begin{pmatrix}
1 ~& 0\\[4pt]
p\e^{-\frac{\ii z^2}{2}}(-z)^{-2\ii\nu(p)} ~& 1
\end{pmatrix},~~~\quad\quad z\in X_2,\\
&\begin{pmatrix}
1 ~& \bar{p}\e^{\frac{\ii z^2}{2}}(-z)^{2\ii\nu(p)}\\[4pt]
0 ~& 1
\end{pmatrix},~~~~~~\qquad z\in X_3,\\
&\begin{pmatrix}
1 ~& 0\\[4pt]
-\frac{p}{1+|p|^2}\e^{-\frac{\ii z^2}{2}}(-z)^{-2\ii\nu(p)} ~& 1
\end{pmatrix},~z\in X_4,\\
\end{aligned}
\right.
\eerr
as $t\rightarrow\infty$ if we set $$p=r(k_1).$$ On the other hand, one can verifies that
\be
J^{Y}(p,z)=\sigma_3\overline{J^X(\bar{p},-\bar{z})}\sigma_3,
\ee
which in turn implies, by uniqueness, that
\be
M^{Y}(p,z)=\sigma_3\overline{M^X(\bar{p},-\bar{z})}\sigma_3,
\ee
where $M^{Y}(p,z)$ is the unique solution of the following RH problem
\be\left\{
\begin{aligned}
&M^{Y}_+(p,z)=M^{Y}_-(p,z)J^{Y}(p,z),~\text{for~almost~every}~z\in X,\\
&M^{Y}(p,z)\rightarrow I,~~~\qquad\qquad\qquad\text{as}~z\rightarrow\infty.
\end{aligned}
\right.
\ee
Therefore, we find that
\be
M^{Y}(p,z)=I-\frac{\ii}{z}\begin{pmatrix}
0 ~& \beta^{Y}(p)\\[4pt]
\overline{\beta^{Y}(p)} &~ 0
\end{pmatrix}+O\bigg(\frac{p}{z^2}\bigg),
\ee
where
\be
\beta^{Y}(p)=\sqrt{\nu(p)}\e^{-\ii(\frac{\pi}{4}+\arg p+\arg\Gamma(-\ii\nu(p)))}.
\ee
As a consequence, we can approximate $M^{(3)}(x,t;k)$ in the neighborhood $D_\varepsilon(k_1)$ of $k_1$ by
\be\label{0.6}
M^{(k_1)}(x,t;k)=(\delta_{k_1}^0)^{\sigma_3}M^{Y}(p,z)(\delta_{k_1}^0)^{-\sigma_3}.
\ee

Then we have the following lemmas about the functions $M^{(k_2)}$ and $M^{(k_1)}$, which will be very useful in deriving the accurate asymptotic formula and error bound in next section.
\begin{lemma}\label{lem3}
For each $t>0$ and $\xi\in\mathcal{I}$, the function $M^{(k_2)}(x,t;k)$ defined in \eqref{3.47} is an analytic function of $k\in D_\varepsilon(k_2)\setminus\mathcal{X}_{k_2}^\varepsilon$. Furthermore,
\be\label{3.48}
|M^{(k_2)}(x,t;k)-I|\leq C,\quad t>3,~\xi\in\mathcal{I},~k\in\overline{ D_\varepsilon(k_2)}\setminus\mathcal{X}_{k_2}^\varepsilon.
\ee
On the other hand, across $\mathcal{X}_{k_2}^\varepsilon$, $M^{(k_2)}$ satisfied the jump condition $M_+^{(k_2)}=M_-^{(k_2)}J^{(k_2)}$ with jump matrix $$J^{(k_2)}=(\delta_{k_2}^0)^{\hat{\sigma}_3}J^X,$$ and $J^{(k_2)}$ satisfies the following estimates for $1\leq p\leq\infty$:
\be\label{3.49}
\|J^{(3)}-J^{(k_2)}\|_{L^p(\mathcal{X}_{k_2}^\varepsilon)}\leq Ct^{-\frac{1}{2}-\frac{1}{2p}}\ln t,
\quad t>3,~\xi\in\mathcal{I},
\ee
where $C>0$ is a constant independent of $t,\xi,k$. Moreover, as $t\rightarrow\infty$,
\be\label{3.50}
\|(M^{(k_2)})^{-1}(x,t;k)-I\|_{L^\infty(\partial D_\varepsilon(k_2))}=O(t^{-1/2}),
\ee
and
\be\label{3.51}
\frac{1}{2\pi\ii}\int_{\partial D_\varepsilon(k_2)}((M^{(k_2)})^{-1}(x,t;k)-I)\dd k=-\frac{(\delta_{k_2}^0)^{\hat{\sigma}_3}M^X_1(\xi)}
{\sqrt{8t(\alpha+6\beta k_2)}}+O(t^{-1}),
\ee
where $M^X_1(\xi)$ is defined by
\be\label{3.52}
M^X_1(\xi)=-\ii\begin{pmatrix}
0 ~& \beta^X(q)\\[4pt]
\overline{\beta^X(q)} ~& 0
\end{pmatrix}.
\ee
\end{lemma}
\begin{proof}
The analyticity of $M^{(k_2)}$ is obvious. Since $|\delta_{k_2}^0(z)|=1$, thus, the estimate \eqref{3.48} follows from the definition of $M^{(k_2)}$ in \eqref{3.47} and the estimate \eqref{3.46}.

On the other hand, we have
\berr
J^{(3)}-J^{(k_2)}=(\delta_{k_2}^0)^{\hat{\sigma}_3}(\tilde{J}-J^X),\quad k\in\mathcal{X}_{k_2}^\varepsilon.
\eerr
However, a careful computation as the Lemma 3.35 in \cite{PD}, we conclude that
\be\label{3.53}
\|\tilde{J}-J^X\|_{L^\infty((\mathcal{X}_{k_2}^\varepsilon)_1)}\leq C|\e^{\frac{\ii\gamma}{2}z^2}|t^{-1/2}\ln t,\quad 0<\gamma<\frac{1}{2},~t>3,~\xi\in\mathcal{I},
\ee
for $k\in(\mathcal{X}_{k_2}^\varepsilon)_1$, that is, $z=\sqrt{8t(\alpha+6\beta k_2)}u\e^{\frac{\ii\pi}{4}}$, $0\leq u\leq\varepsilon$.
Thus,
\be\label{3.53'}
\|\tilde{J}-J^X\|_{L^1((\mathcal{X}_{k_2}^\varepsilon)_1)}\leq Ct^{-1}\ln t,\quad t>3,~\xi\in\mathcal{I}.
\ee
By the general inequality $\|f\|_{L^p}\leq\|f\|^{1-1/p}_{L^\infty}\|f\|_{L^1}^{1/p}$, we find
\be
\|\tilde{J}-J^X\|_{L^p((\mathcal{X}_{k_2}^\varepsilon)_1)}\leq Ct^{-1/2-1/2p}\ln t,\quad t>3,~\xi\in\mathcal{I}.
\ee
The norms on $(\mathcal{X}_{k_2}^\varepsilon)_j$, $j=2,3,4$, are estimated in a similar way. Therefore, \eqref{3.49} follows.

If $k\in\partial D_\varepsilon(k_2)$, the variable $z=\sqrt{8t(\alpha+6\beta k_2)}(k-k_2)$ tends to infinity as $t\rightarrow\infty$. It follows from \eqref{3.43} that
\berr
M^X(q,z)=I+\frac{M^X_1(\xi)}{\sqrt{8t(\alpha+6\beta k_2)}(k-k_2)}+O\bigg(\frac{q}{t}\bigg),\quad t\rightarrow\infty,~k\in \partial D_\varepsilon(k_2),
\eerr
where $M^X_1(\xi)$ is defined by \eqref{3.52}. Since
 $$M^{(k_2)}(x,t;k)=(\delta_{k_2}^0)^{\hat{\sigma}_3}M^X(q,z),$$
thus we have
\be\label{3.54}
(M^{(k_2)})^{-1}(x,t;k)-I=-\frac{(\delta_{k_2}^0)^{\hat{\sigma}_3}M^X_1(\xi)}
{\sqrt{8t(\alpha+6\beta k_2)}(k-k_2)}+O\bigg(\frac{q}{t}\bigg),\quad t\rightarrow\infty,~k\in \partial D_\varepsilon(k_2).
\ee
The estimate \eqref{3.50} immediately follows from \eqref{3.54} and $|M_1^X|\leq C$. By Cauchy's formula and \eqref{3.54}, we derive \eqref{3.51}.
\end{proof}

Accordingly, we have the following lemma for the function $M^{(k_1)}(x,t;k)$.
\begin{lemma}\label{lem34}
For each $t>0$ and $\xi\in\mathcal{I}$, the function $M^{(k_1)}(x,t;k)$ defined in \eqref{0.6} is an analytic function of $k\in D_\varepsilon(k_1)\setminus\mathcal{X}_{k_1}^\varepsilon$. Furthermore,
\be\label{0.7}
|M^{(k_1)}(x,t;k)-I|\leq C,\quad t>3,~\xi\in\mathcal{I},~k\in\overline{ D_\varepsilon(k_1)}\setminus\mathcal{X}_{k_1}^\varepsilon.
\ee
Across $\mathcal{X}_{k_1}^\varepsilon$, $M^{(k_1)}$ satisfied the jump condition $M_+^{(k_1)}=M_-^{(k_1)}J^{(k_1)}$ with jump matrix $$J^{(k_1)}=(\delta_{k_1}^0)^{\hat{\sigma}_3}J^Y,$$ and $J^{(k_1)}$ satisfies the following estimates:
\be\label{0.8}
\|J^{(3)}-J^{(k_1)}\|_{L^p(\mathcal{X}_{k_1}^\varepsilon)}\leq Ct^{-1/2-1/2p}\ln t,
\quad t>3,~\xi\in\mathcal{I},
\ee
for $1\leq p\leq\infty$, where $C>0$ is a constant independent of $t,\xi,k$. Moreover, as $t\rightarrow\infty$,
\be\label{0.9}
\|(M^{(k_1)})^{-1}(x,t;k)-I\|_{L^\infty(\partial D_\varepsilon(k_1))}=O(t^{-1/2}),
\ee
and
\be\label{0.10}
\frac{1}{2\pi\ii}\int_{\partial D_\varepsilon(k_1)}((M^{(k_1)})^{-1}(x,t;k)-I)\dd k=-\frac{(\delta_{k_1}^0)^{\hat{\sigma}_3}M^Y_1(\xi)}
{\sqrt{-8t(\alpha+6\beta k_1)}}+O(t^{-1}),
\ee
where $M^Y_1(\xi)$ is given by
\be\label{0.11}
M^Y_1(\xi)=-\ii\begin{pmatrix}
0 ~& \beta^Y(p)\\[4pt]
\overline{\beta^Y(p)} ~& 0
\end{pmatrix}.
\ee
\end{lemma}

\section{Derivation of the long-time asymptotic formula and error bound}
\setcounter{equation}{0}
\setcounter{lemma}{0}
\setcounter{theorem}{0}

We now begin to establish the explicit long-time asymptotic formula for the Hirota equation \eqref{1.1} on the half-line.

Define the approximate solution $M^{(app)}(x,t;k)$ by
\be\label{4.58}
M^{(app)}=\left\{\begin{aligned}
&M^{(k_1)},\quad k\in D_\varepsilon(k_1),\\
&M^{(k_2)},\quad k\in D_\varepsilon(k_2),\\
&I,\qquad ~~~{\text elsewhere}.
\end{aligned}
\right.
\ee
Let $\hat{M}(x,t;k)$ be
\be\label{4.58'}
\hat{M}=M^{(3)}(M^{(app)})^{-1},
\ee
 then $\hat{M}(x,t;k)$ satisfies the following RH problem
\be\label{4.59}
\hat{M}_+(x,t;k)=\hat{M}_-(x,t;k)\hat{J}(x,t,k),\quad k\in\hat{\Sigma},
\ee
where the jump contour $\hat{\Sigma}=\Sigma^{(3)}\cup\partial D_\varepsilon(k_1)\cup\partial D_\varepsilon(k_2)$ is depicted in Fig. \ref{fig7}, and the jump matrix $\hat{J}(x,t,k)$ is given by
\be\label{4.60}
\hat{J}=\left\{
\begin{aligned}
&M^{(app)}_-J^{(3)}(M^{(app)}_+)^{-1},~~ k\in\hat{\Sigma}\cap (D_\varepsilon(k_1)\cup D_\varepsilon(k_2)),\\
&(M^{(app)})^{-1},\qquad\qquad\quad k\in(\partial D_\varepsilon(k_1)\cup\partial D_\varepsilon(k_2)),\\
&J^{(3)},\qquad\qquad\qquad\quad~~ k\in\hat{\Sigma}\setminus (\overline{D_\varepsilon(k_1)}\cup\overline{D_\varepsilon(k_2)}).
\end{aligned}
\right.
\ee

\begin{figure}[htbp]
\begin{minipage}[t]{0.3\linewidth}
\centering
\includegraphics[width=4in]{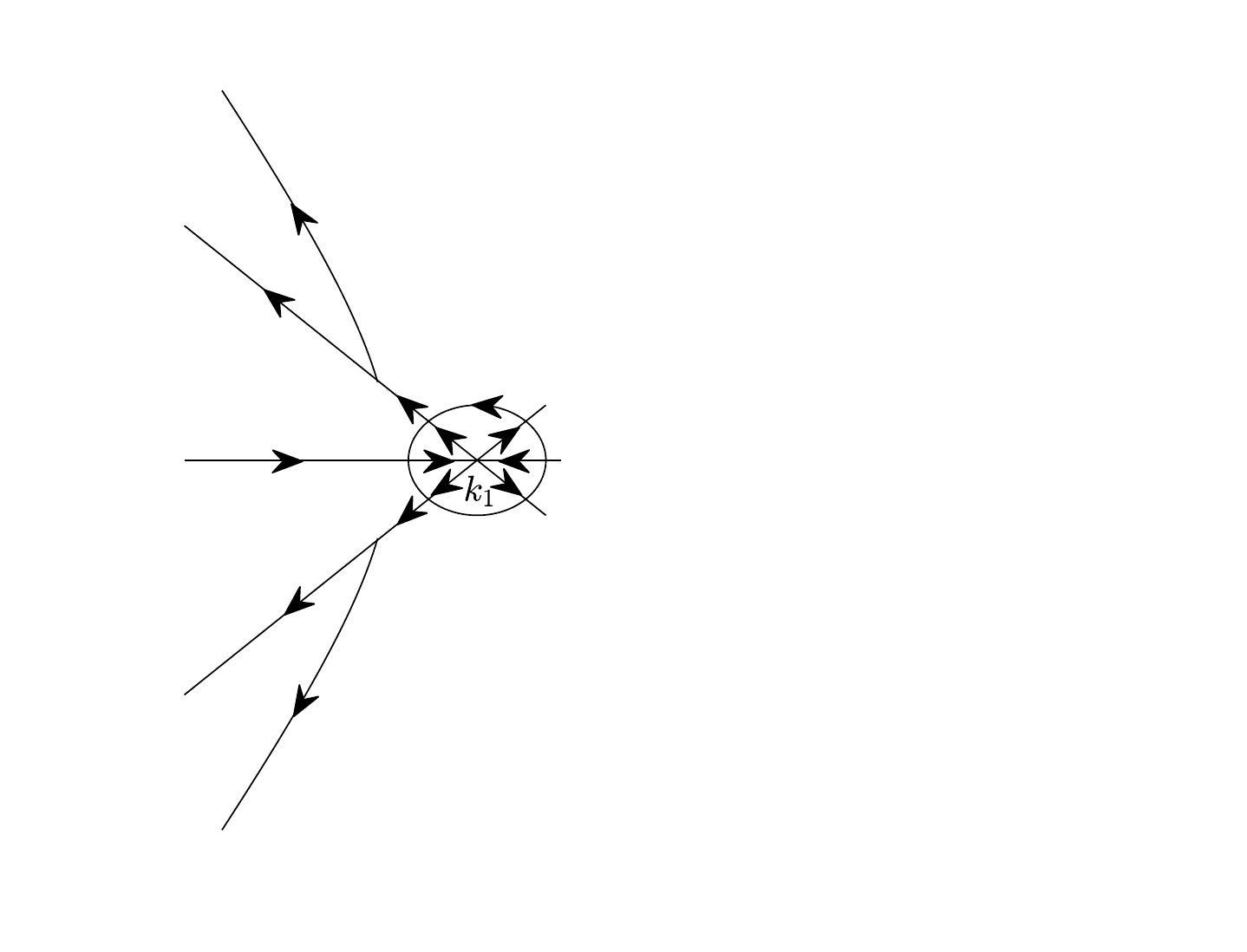}
\label{fig:side:a}
\end{minipage}%
\begin{minipage}[t]{0.1\linewidth}
\centering
\includegraphics[width=4in]{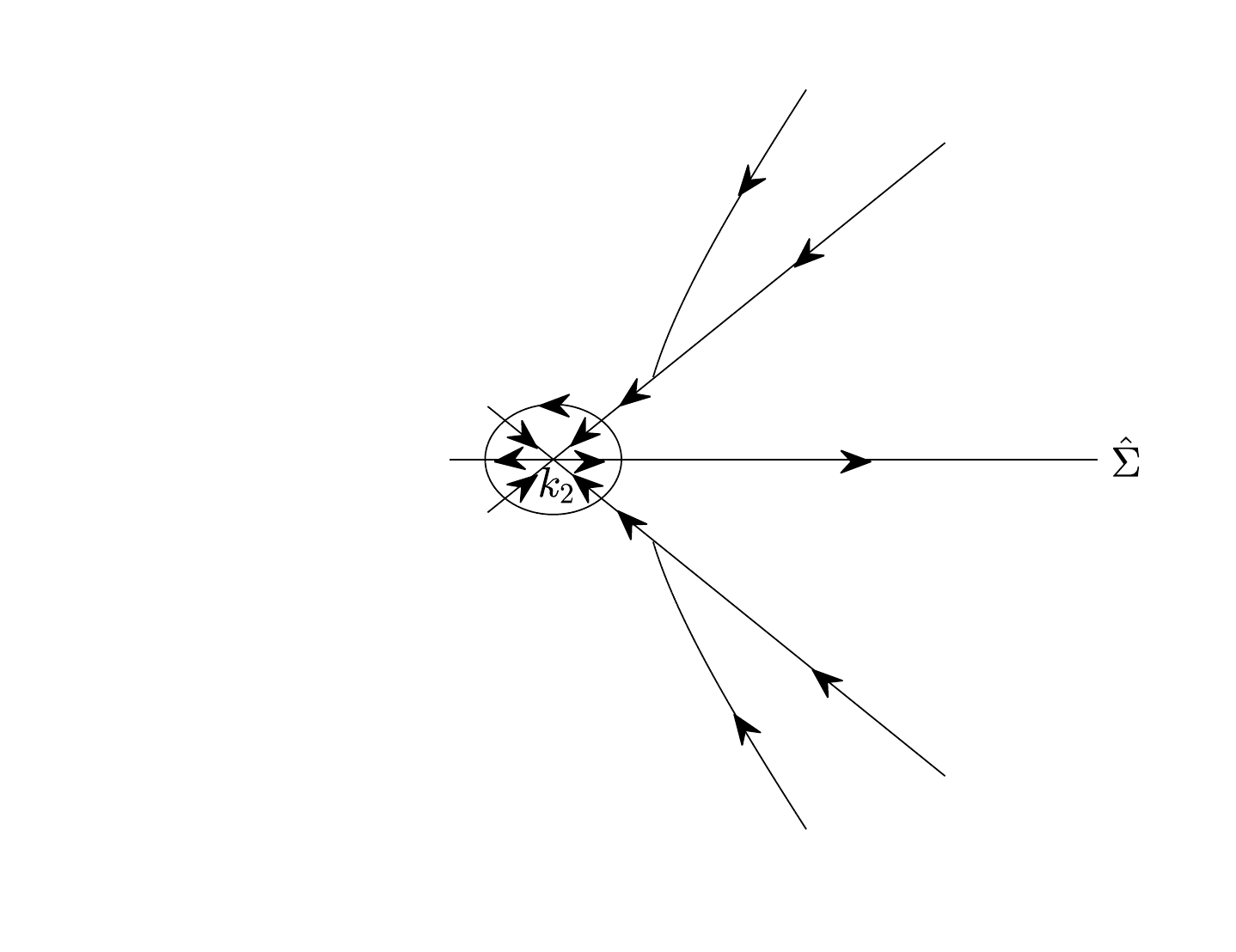}
\label{fig:side:b}
\end{minipage}
\caption{The contour $\hat{\Sigma}$. We separate the enlarged circles $D_\varepsilon(k_1)$ and $D_\varepsilon(k_2)$ to show detail.}\label{fig7}
\end{figure}

For convenience, we rewrite $\hat{\Sigma}$ as follows:
$$\hat{\Sigma}=(\partial D_\varepsilon(k_1)\cup\partial D_\varepsilon(k_2))\cup(\mathcal{X}_{k_1}^\varepsilon\cup\mathcal{X}_{k_2}^\varepsilon)\cup\hat{\Sigma}_1\cup\hat{\Sigma}_2,$$
where
$$\hat{\Sigma}_1=\bigcup_1^6\Sigma^{(3)}_j\setminus(D_\varepsilon(k_1)\cup D_\varepsilon(k_2)),~\hat{\Sigma}_2=\bigcup_7^{10}\Sigma^{(3)}_j,$$
and $\{\Sigma^{(3)}_j\}_1^{10}$ denoting the restriction of $\Sigma^{(3)}$ to the contour labeled by $j$ in Fig. \ref{fig4}.
 Then we have the following lemma if we let $\hat{w}=\hat{J}-I$.
\begin{lemma}\label{lem4}
For $1\leq p\leq\infty$, the following estimates hold for $t>3$ and $\xi\in\mathcal{I}$,
\bea
\|\hat{w}\|_{L^p(\partial D_\varepsilon(k_1)\cup\partial D_\varepsilon(k_2))}&\leq& Ct^{-1/2},\label{4.64}\\
\|\hat{w}\|_{L^p(\mathcal{X}_{k_1}^\varepsilon\cup\mathcal{X}_{k_2}^\varepsilon)}&\leq&Ct^{-\frac{1}{2}-\frac{1}{2p}}\ln t,\label{4.62}\\
\|\hat{w}\|_{L^p(\hat{\Sigma}_1)}&\leq&C\e^{-ct},\label{4.61}\\
\|\hat{w}\|_{L^p(\hat{\Sigma}_2)}&\leq&Ct^{-3/2}.\label{4.63}
\eea
\end{lemma}
\begin{proof}
The inequality \eqref{4.64} is a consequence of \eqref{3.50}, \eqref{0.9} and \eqref{4.60}.

For $k\in\mathcal{X}_{k_j}^\varepsilon$, we find $$\hat{w}=M^{(k_j)}_-(J^{(3)}-J^{(k_j)})(M^{(k_j)}_+)^{-1},\quad j=1,2.$$ Therefore, it follows from \eqref{3.49} and \eqref{0.8} that the estimate \eqref{4.62} holds.

For $k\in D_1\cap\hat{\Sigma}_1$, let $k=k_2+u\e^{\frac{\ii\pi}{4}}$, $\varepsilon<u<\infty$, then
\berr
\text{Re}\Phi(k)=-4u^2(\sqrt{2}\beta u+\sqrt{\alpha^2-3\beta\xi}).
\eerr
Since $\hat{w}$ only has a nonzero $-(r_{1,a}+h)\delta^{-2}\e^{t\Phi}$ in $(21)$ entry, hence, for $t\geq1$, by \eqref{3.12} and \eqref{3.30}, we get
\bea
|\hat{w}_{21}|&=&|-(r_{1,a}+h)\delta^{-2}\e^{t\Phi}|\nn\\
&\leq& C|r_{1,a}+h_a|\e^{-t|\text{Re}\Phi|}\nn\\
&\leq&\frac{C}{1+|k|^2}\e^{-\frac{3t}{4}|\text{Re}\Phi|}\leq  C\e^{-c\varepsilon^2t}.\nn
\eea
In a similar way, the other estimates on $\hat{\Sigma}_1$ hold.
This proves \eqref{4.61}.

Since the matrix $\hat{w}$ on $\hat{\Sigma}_2$ only involves the small remainders $h_r$, $r_{1,r}$ and $r_{2,r}$, thus, by Lemmas \ref{lem1} and \ref{lem2}, the estimate \eqref{4.63} follows.
\end{proof}
The estimates in Lemma \ref{lem4} imply that
\be\label{5.9}
\begin{aligned}
\|\hat{w}\|_{(L^1\cap L^2)(\hat{\Sigma})}&\leq Ct^{-1/2},\\
\|\hat{w}\|_{L^\infty(\hat{\Sigma})}&\leq Ct^{-1/2}\ln t,
\end{aligned}
\quad t>3,~ \xi\in\mathcal{I}.
\ee
Let $\hat{C}$ denote the Cauchy operator associated with $\hat{\Sigma}$:
\berr
(\hat{C}f)(k)=\int_{\hat{\Sigma}}\frac{f(\zeta)}{\zeta-k}\frac{\dd \zeta}{2\pi\ii},\quad k\in\bfC\setminus\hat{\Sigma},~f\in L^2(\hat{\Sigma}).
\eerr
We denote the boundary values of $\hat{C}f$ from the left and right sides of $\hat{\Sigma}$ by $\hat{C}_+f$ and $\hat{C}_-f$, respectively. As is well known, the operators $\hat{C}_\pm$ are bounded from $L^2(\hat{\Sigma})$ to $L^2(\hat{\Sigma})$, and $\hat{C}_+-\hat{C}_-=I$, here $I$ denotes the identity operator.

Define the operator $\hat{C}_{\hat{w}}$: $L^2(\hat{\Sigma})+L^\infty(\hat{\Sigma})\rightarrow L^2(\hat{\Sigma})$ by $\hat{C}_{\hat{w}}f=\hat{C}_-(f\hat{w}),$ that is, $\hat{C}_{\hat{w}}$ is defined by $\hat{C}_{\hat{w}}(f)=\hat{C}_+(f\hat{w}_-)+\hat{C}_-(f\hat{w}_+)$ where we have chosen, for simplicity, $\hat{w}_+=\hat{w}$ and $\hat{w}_-=0$. Then, by \eqref{5.9}, we find
\be\label{4.66}
\|\hat{C}_{\hat{w}}\|_{B(L^2(\hat{\Sigma}))}\leq C\|\hat{w}\|_{L^\infty(\hat{\Sigma})}\leq Ct^{-1/2}\ln t,
\ee
where $B(L^2(\hat{\Sigma}))$ denotes the Banach space of bounded linear operators $L^2(\hat{\Sigma})\rightarrow L^2(\hat{\Sigma})$. Therefore, there exists a $T>0$ such that $I-\hat{C}_{\hat{w}}\in B(L^2(\hat{\Sigma}))$ is invertible for all $\xi\in\mathcal{I},$ $t>T$. Following this, we may define the $2\times2$ matrix-valued function $\hat{\mu}(x,t;k)$ whenever $t>T$ by
\be\label{4.67}
\hat{\mu}=I+\hat{C}_{\hat{w}}\hat{\mu}.
\ee
Then
\be\label{4.68}
\hat{M}(x,t;k)=I+\frac{1}{2\pi\ii}\int_{\hat{\Sigma}}\frac{(\hat{\mu}\hat{w})(x,t;\zeta)}{\zeta-k}\dd\zeta,\quad k\in\bfC\setminus\hat{\Sigma}
\ee
is the unique solution of the RH problem \eqref{4.59} for $t>T$.

Moreover, the function $\hat{\mu}(x,t;k)$ satisfies
\be\label{4.69}
\|\hat{\mu}(x,t;\cdot)-I\|_{L^2(\hat{\Sigma})}=O(t^{-1/2}),\quad t\rightarrow\infty,~\xi\in\mathcal{I}.
\ee
In fact, equation \eqref{4.67} is equivalent to $\hat{\mu}=I+(I-\hat{C}_{\hat{w}})^{-1}\hat{C}_{\hat{w}}I$. Using the Neumann series, we get
$$\|(I-\hat{C}_{\hat{w}})^{-1}\|_{B(L^2(\hat{\Sigma}))}\leq\frac{1}{1-\|\hat{C}_{\hat{w}}\|_{B(L^2(\hat{\Sigma}))}}$$
whenever $\|\hat{C}_{\hat{w}}\|_{B(L^2(\hat{\Sigma}))}<1$. Thus, we find
\bea
\|\hat{\mu}(x,t;\cdot)-I\|_{L^2(\hat{\Sigma})}&=&\|(I-\hat{C}_{\hat{w}})^{-1}\hat{C}_{\hat{w}}I\|_{L^2(\hat{\Sigma})}\nn\\
&\leq&\|(I-\hat{C}_{\hat{w}})^{-1}\|_{B(L^2(\hat{\Sigma}))}\|\hat{C}_-(\hat{w})\|_{L^2(\hat{\Sigma})}\nn\\
&\leq&\frac{C\|\hat{w}\|_{L^2(\hat{\Sigma})}}{1-\|\hat{C}_{\hat{w}}\|_{B(L^2(\hat{\Sigma}))}}\leq C\|\hat{w}\|_{L^2(\hat{\Sigma})}\nn
\eea
for all $t$ large enough and all $\xi\in\mathcal{I}$. In view of \eqref{5.9}, this gives \eqref{4.69}.

It follows from \eqref{4.68} that
\be\label{4.70}
\lim_{k\rightarrow\infty}k(\hat{M}(x,t;k)-I)=-\frac{1}{2\pi\ii}\int_{\hat{\Sigma}}(\hat{\mu}\hat{w})(x,t;k)\dd k.
\ee
Using \eqref{4.61} and \eqref{4.69}, we have
\bea
\int_{\hat{\Sigma}_1}(\hat{\mu}\hat{w})(x,t;k)\dd k&=&\int_{\hat{\Sigma}_1}\hat{w}(x,t;k)\dd k+\int_{\hat{\Sigma}_1}(\hat{\mu}(x,t;k)-I)\hat{w}(x,t;k)\dd k\nn\\
&\leq&\|\hat{w}\|_{L^1(\hat{\Sigma}_1)}+\|\hat{\mu}-I\|_{L^2(\hat{\Sigma}_1)}\|\hat{w}\|_{L^2(\hat{\Sigma}_1)}\nn\\
&\leq&C\e^{-ct},\quad t\rightarrow\infty.\nn
\eea
Similarly, by \eqref{4.62} and \eqref{4.69}, the contribution from $\mathcal{X}_{k_1}^\varepsilon\cup\mathcal{X}_{k_2}^\varepsilon$ to the right-hand side of \eqref{4.70} is $$O(\|\hat{w}\|_{L^1(\mathcal{X}_{k_1}^\varepsilon\cup\mathcal{X}_{k_2}^\varepsilon)}+\|\hat{\mu}-I\|_{L^2(\mathcal{X}_{k_1}^\varepsilon\cup\mathcal{X}_{k_2}^\varepsilon)}
\|\hat{w}\|_{L^2(\mathcal{X}_{k_1}^\varepsilon\cup\mathcal{X}_{k_2}^\varepsilon)})=O(t^{-1}\ln t),\quad t\rightarrow\infty.$$
By \eqref{4.63} and \eqref{4.69}, the contribution from $\hat{\Sigma}_2$ to the right-hand side of \eqref{4.70} is
\berr
O(\|\hat{w}\|_{L^1(\hat{\Sigma}_2)}+\|\hat{\mu}-I\|_{L^2(\hat{\Sigma}_2)}
\|\hat{w}\|_{L^2(\hat{\Sigma}_2)})=O(t^{-3/2}),\quad t\rightarrow\infty.
\eerr
Finally, by \eqref{3.51}, \eqref{0.10}, \eqref{4.64} and \eqref{4.69}, we can get
\bea
&&-\frac{1}{2\pi\ii}\int_{\partial D_\varepsilon(k_1)\cup\partial D_\varepsilon(k_2)}(\hat{\mu}\hat{w})(x,t;k)\dd k\nn\\
&=&-\frac{1}{2\pi\ii}\int_{\partial D_\varepsilon(k_1)\cup\partial D_\varepsilon(k_2)}\hat{w}(x,t;k)\dd k-\frac{1}{2\pi\ii}\int_{\partial D_\varepsilon(k_1)\cup\partial D_\varepsilon(k_2)}(\hat{\mu}(x,t;k)-I)\hat{w}(x,t;k)\dd k\nn\\
&=&-\frac{1}{2\pi\ii}\int_{\partial D_\varepsilon(k_1)}\bigg((M^{(k_1)})^{-1}(x,t;k)-I\bigg)\dd k-\frac{1}{2\pi\ii}\int_{\partial D_\varepsilon(k_2)}\bigg((M^{(k_2)})^{-1}(x,t;k)-I\bigg)\dd k\nn\\
&&+O(\|\hat{\mu}-I\|_{L^2(\partial D_\varepsilon(k_1)\cup\partial D_\varepsilon(k_2))}
\|\hat{w}\|_{L^2(\partial D_\varepsilon(k_1)\cup\partial D_\varepsilon(k_2))})\nn\\
&=&\frac{(\delta_{k_1}^0)^{\hat{\sigma}_3}M^Y_1(\xi)}
{\sqrt{-8t(\alpha+6\beta k_1)}}+\frac{(\delta_{k_2}^0)^{\hat{\sigma}_3}M^X_1(\xi)}
{\sqrt{8t(\alpha+6\beta k_2)}}+O(t^{-1}),\quad t\rightarrow\infty.\nn
\eea
Thus, we obtain the following important relation
\be\label{4.71}
\lim_{k\rightarrow\infty}k(\hat{M}(x,t;k)-I)=\frac{(\delta_{k_1}^0)^{\hat{\sigma}_3}M^Y_1(\xi)}
{\sqrt{-8t(\alpha+6\beta k_1)}}+\frac{(\delta_{k_2}^0)^{\hat{\sigma}_3}M^X_1(\xi)}
{\sqrt{8t(\alpha+6\beta k_2)}}+O(t^{-1}\ln t),\quad t\rightarrow\infty.
\ee

Taking into account that \eqref{2.19}, \eqref{3.5}, \eqref{3.7}, \eqref{3.9}, \eqref{3.31}, \eqref{3.32} and \eqref{4.58'}, for sufficient large $k\in\bfC\setminus\hat{\Sigma}$, we get
\bea\label{4.72}
u(x,t)&=&2\ii\lim_{k\rightarrow\infty}(kM(x,t;k))_{12}\nn\\
&=&2\ii\lim_{k\rightarrow\infty}k(\hat{M}(x,t;k)-I)_{12}\nn\\
&=&2\ii\bigg(\frac{-\ii\beta^Y(p)(\delta_{k_1}^0)^{2}}
{\sqrt{-8t(\alpha+6\beta k_1)}}+\frac{-\ii\beta^X(q)(\delta_{k_2}^0)^2}{\sqrt{8t(\alpha+6\beta k_2)}}\bigg)+O\bigg(\frac{\ln t}{t}\bigg).\nn
\eea

Collecting the above computations, we obtain our main results stated as the following theorem.
\begin{theorem}\label{the4.2}
Let $u_0(x),g_0(t),g_1(t),g_2(t)$ lie in the Schwartz space $S([0,\infty))$. Suppose the assumption \ref{ass1} be valid. Then, for any positive constant $N$, as $t\rightarrow\infty$, the solution $u(x,t)$ of the IBV problem for Hirota equation \eqref{1.1} on the half-line satisfies the following asymptotic formula
\be
u(x,t)=\frac{u_{as}(x,t)}{\sqrt{t}}+O\bigg(\frac{\ln t}{t}\bigg),\quad t\rightarrow\infty,~\xi=\frac{x}{t}\in\mathcal{I},
\ee
where the error term is uniform with respect to $x$ in the given range, and the leading-order
coefficient $u_{as}(x,t)$ is defined by
\be
u_{as}(x,t)=\sqrt{\frac{\nu(k_1)}{-2(\alpha+6\beta k_1)}}\e^{\ii\phi_a(\xi,t)}+\sqrt{\frac{\nu(k_2)}{2(\alpha+6\beta k_2)}}\e^{\ii\phi_b(\xi,t)},
\ee
where
\bea
\phi_a(\xi,t)&=&-\frac{\pi}{4}-\arg r(k_1)+\arg\Gamma(\ii\nu(k_1))-\nu(k_1)\ln(-8tk_1^2(\alpha+6\beta k_1))+4k_1^2t(\alpha+4\beta k_1)\nn\\
&&-\frac{1}{\pi}\int_{k_1}^{k_2}\ln\bigg(\frac{1+|r(s)|^2}{1+|r(k_1)|^2}\bigg)\frac{\dd s}{s-k_1},\nn\\
\phi_b(\xi,t)&=&\frac{\pi}{4}-\arg r(k_2)-\arg\Gamma(\ii\nu(k_2))+\nu(k_2)\ln(8tk_1^2(\alpha+6\beta k_2))+4k_2^2t(\alpha+4\beta k_2)\nn\\
&&-\frac{1}{\pi}\int_{k_1}^{k_2}\ln\bigg(\frac{1+|r(s)|^2}{1+|r(k_2)|^2}\bigg)\frac{\dd s}{s-k_2},\nn
\eea
and $k_1,~k_2$, $\nu(k_1),~\nu(k_2)$ are defined by \eqref{3.3}, \eqref{3.4}, \eqref{3.35} and \eqref{3.35'}, respectively.
\end{theorem}

\medskip
\small{

}
\end{document}